\documentclass[11pt,letterpaper]{article}
\usepackage{odonnell}

\newcommand{\symm}[1]{\mathfrak{S}_{#1}}
\newcommand{\nlift}{N}
\newcommand{\atms}{c}
\newcommand{\vtcs}{n}
\newcommand{\repdim}{d}
\newcommand{\atom}{atom\xspace}
\newcommand{\atoms}{atoms\xspace}

\newcommand{\Unitaries}[1]{\mathrm{U}(#1)}

\newcommand{\Hyperoctahedrals}[1]{\mathrm{H}(#1)} %B is also traditional

\newcommand{\sgnrep}{\mathtt{sgn}}
\newcommand{\stdrep}{\mathtt{std}}
\newcommand{\permrep}{\mathtt{perm}}
\newcommand{\idrep}{\mathtt{def}}
\newcommand{\ketket}[2]{\ket{#1}\!\ket{#2}} % fix this later
\newcommand{\charpoly}[2]{\phi\parens*{#1;#2}}
\newcommand{\matchpoly}[2]{\mu\parens*{#1;#2}}
\newcommand{\factorpoly}[2]{\alpha\parens*{#1;#2}} \newcommand{\factorpolyplain}{\alpha}
\newcommand{\factorpolynomial}{additive characteristic polynomial\xspace}
\newcommand{\additiveproduct}{additive product\xspace}
\newcommand{\addprod}{\pluscirc}
\newcommand{\roots}{\mathrm{roots}}

\newcommand{\mathfunc}{\mathsf}
\newcommand{\CCyc}{\mathfunc{CCyc}}
\newcommand{\FrWalk}{\mathfunc{FrWalk}}
\newcommand{\Heap}{\mathfunc{Heap}}
\newcommand{\THeap}{\mathfunc{THeap}}
\newcommand{\Pyr}{\mathfunc{Pyr}}
\newcommand{\length}{\ell}
\newcommand{\spec}{\mathfunc{spec}}
\newcommand{\specrad}{\rho}
\newcommand{\IG}{X}     % infinite graph
\newcommand{\ITree}{\mathbb{T}}
\newcommand{\FinQuo}{\mathfunc{FinQuo}}
\newcommand{\UCT}{\mathfunc{UCT}}
\newcommand{\coloredcyc}{\widetilde{\upgamma}}
\newcommand{\freelikewalk}{\widetilde{\upomega}}
\newcommand{\Matchings}{\mathfunc{Matchings}}

\newcommand{\Type}{\mathfunc{Type}}
\newcommand{\maxroot}{\mathfunc{maxroot}}
\newcommand{\BitToSwap}{\mathfunc{Swap}}
\newcommand{\StringToPerm}{\mathfunc{Perm}}
\newcommand{\StringToPotential}{\mathfunc{Potential}}
\newcommand{\LiftEnc}{\mathfunc{LiftEnc}}
\newcommand{\Adj}{\mathfunc{Adj}}
\newcommand{\LiftAdj}[2]{\mathfunc{LiftAdj}_{#1,#2}}
\newcommand{\freeprods}{\ell}

\begin{document}

\title{$X$-Ramanujan graphs}
\author{Sidhanth Mohanty\thanks{Electrical Engineering and Computer Sciences Department, University of California Berkeley.  \texttt{sidhanthm@cs.berkeley.edu}.  A substantial portion of this work was done while the author was at Carnegie Mellon University.} \and Ryan O'Donnell\thanks{Computer Science Department, Carnegie Mellon University. \texttt{odonnell@cs.cmu.edu}. Supported by NSF grants CCF-1618679 and CCF-1717606. This material is based upon work supported by the National Science Foundation under grant numbers listed above. Any opinions, findings and conclusions or recommendations expressed in this material are those of the author and do not necessarily reflect the views of the National Science Foundation (NSF).}}
\date{\today}

\maketitle

\begin{abstract}
Let $X$ be an infinite graph of bounded degree; e.g., the Cayley graph of a free product of finite groups.  If $G$ is a finite graph covered by~$X$, it is said to be \emph{$X$-Ramanujan} if its second-largest eigenvalue $\lambda_2(G)$ is at most the spectral radius $\specrad(X)$ of~$X$, and more generally $k$-quasi-$X$-Ramanujan if $\lambda_k(G)$ is at most $\specrad(X)$.  In case $X$ is the infinite $\Delta$-regular tree, this reduces to the well known notion of a finite $\Delta$-regular graph being \emph{Ramanujan}.  Inspired by the Interlacing Polynomials method of Marcus, Spielman, and Srivastava, we show the existence of infinitely many $k$-quasi-$X$-Ramanujan graphs for a variety of infinite~$X$.  In particular, $X$ need not be a tree; our analysis is applicable whenever $X$ is what we call an \emph{\additiveproduct} graph. This \additiveproduct is a new construction of an infinite graph $A_1 \addprod \cdots \addprod A_\atms$ from finite ``\atom'' graphs $A_1, \dots, A_\atms$ over a common vertex set.  It generalizes the notion of the free product graph $A_1 \ast \cdots \ast A_\atms$ when the \atoms $A_j$ are vertex-transitive, and it generalizes the notion of the universal covering tree when the \atoms $A_j$ are single-edge graphs.  Key to our analysis is a new graph polynomial $\factorpoly{A_1, \dots, A_\atms}{x}$ that we call the \emph{\factorpolynomial}. It generalizes the well known matching polynomial $\matchpoly{G}{x}$ in case the \atoms $A_j$ are the single edges of~$G$, and it generalizes the $r$-characteristic polynomial introduced in~\cite{Rav16,LR18}.  We show that $\factorpoly{A_1, \dots, A_\atms}{x}$ is real-rooted, and all of its roots have magnitude at most $\specrad(A_1 \addprod \cdots \addprod A_\atms)$.  This last fact is proven by generalizing Godsil's notion of treelike walks on a graph~$G$ to a notion of \emph{freelike walks} on a collection of \atoms $A_1, \dots, A_\atms$.\\

A talk about this work may be viewed \href{https://youtu.be/8yNH_0Pf6U4}{on YouTube}.
\end{abstract}

\setcounter{page}{0}
\thispagestyle{empty}
\newpage

\section{Introduction}
Let $X$ be an infinite graph, such as one of the eight partially sketched below.
\begin{figure}[H]
  \centering
  \includegraphics[width=.28\textwidth]{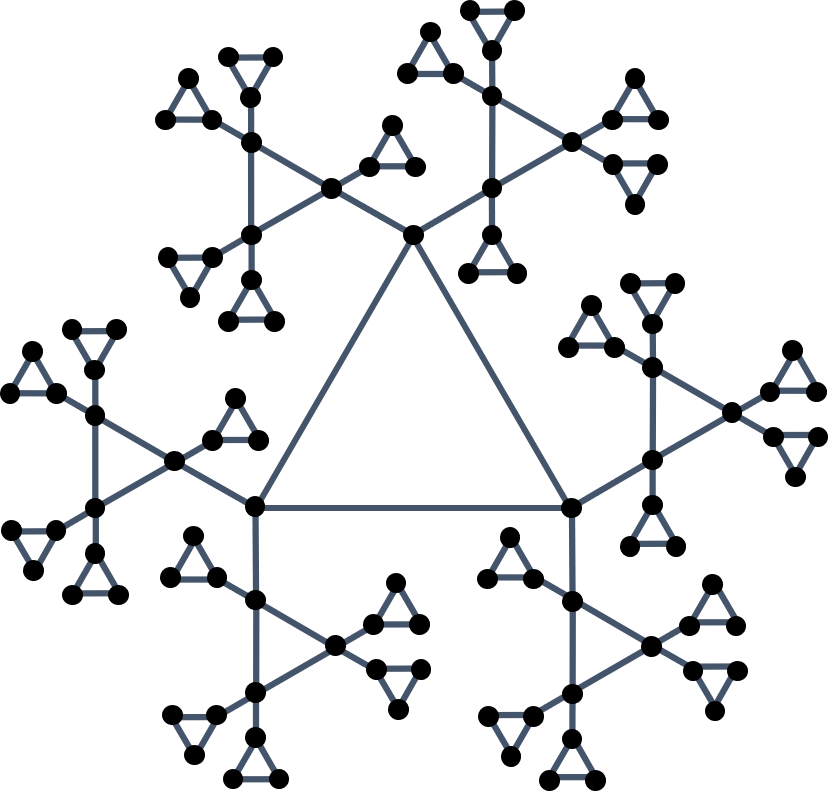}
  \qquad \qquad
  \includegraphics[width=.28\textwidth]{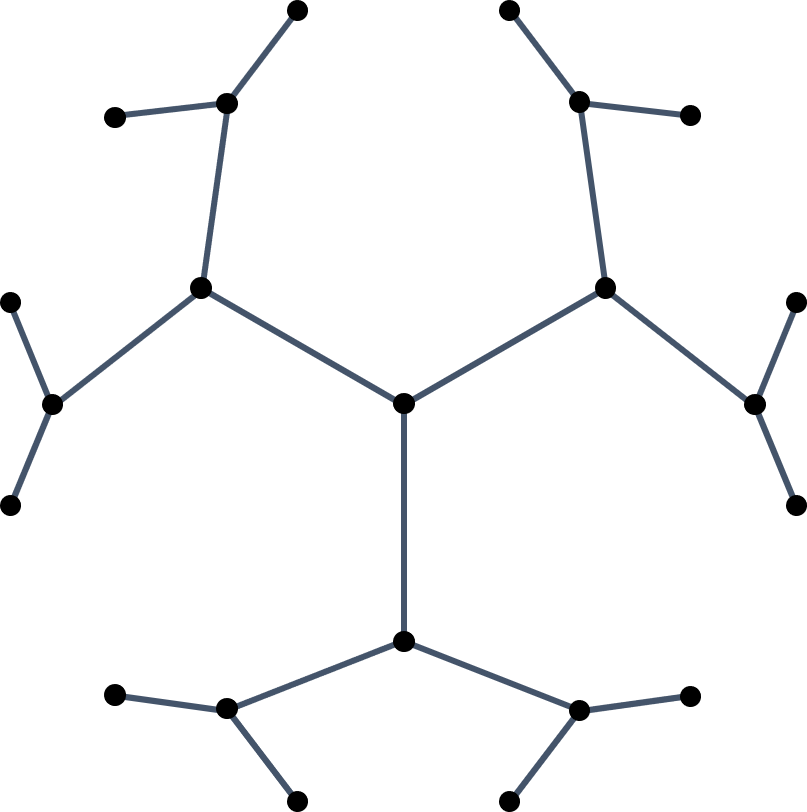} \\[.15in]

  \includegraphics[width=.28\textwidth]{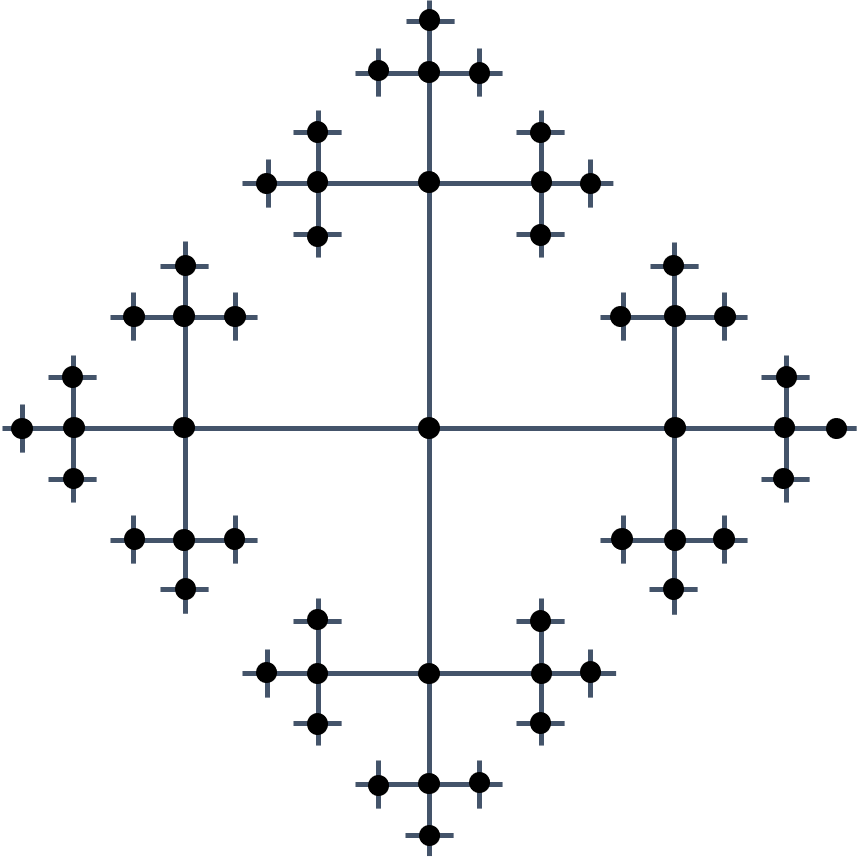}
  \qquad \qquad
  \includegraphics[width=.28\textwidth]{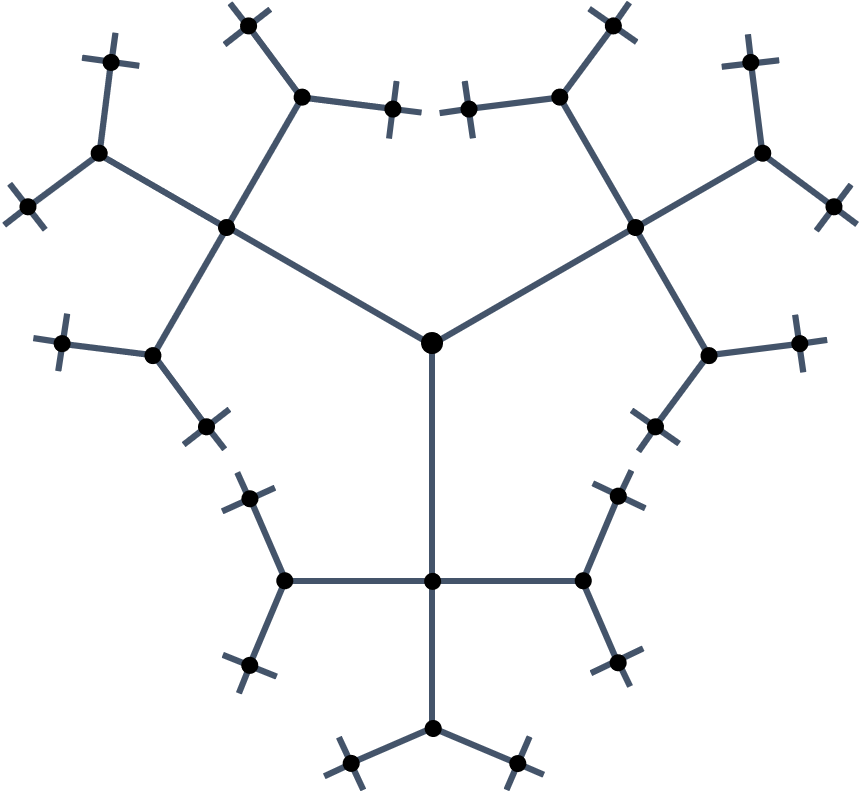}\\[.15in]

  \includegraphics[width=.28\textwidth]{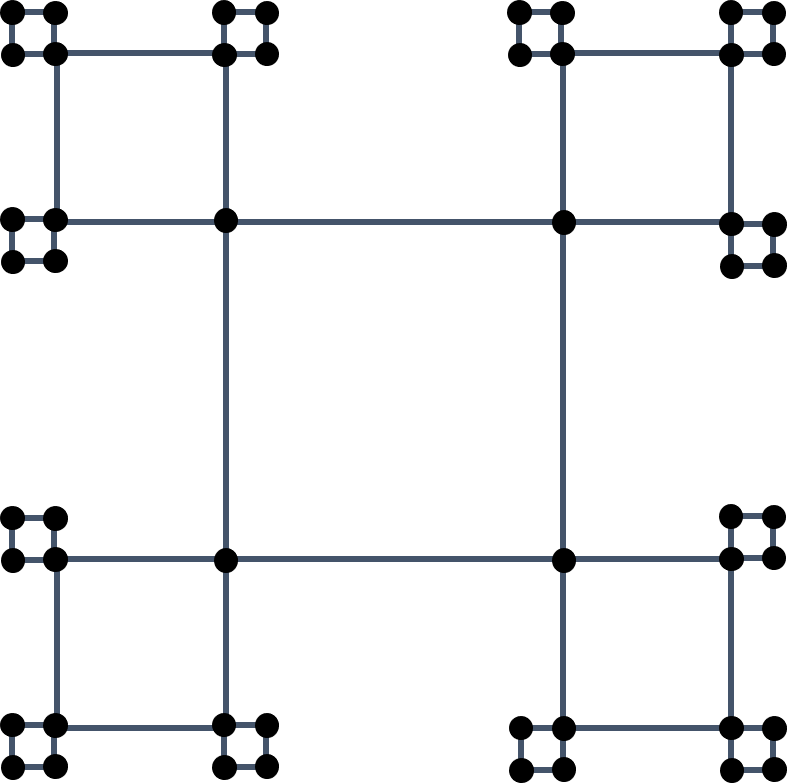}
  \qquad \qquad
  \includegraphics[width=.28\textwidth]{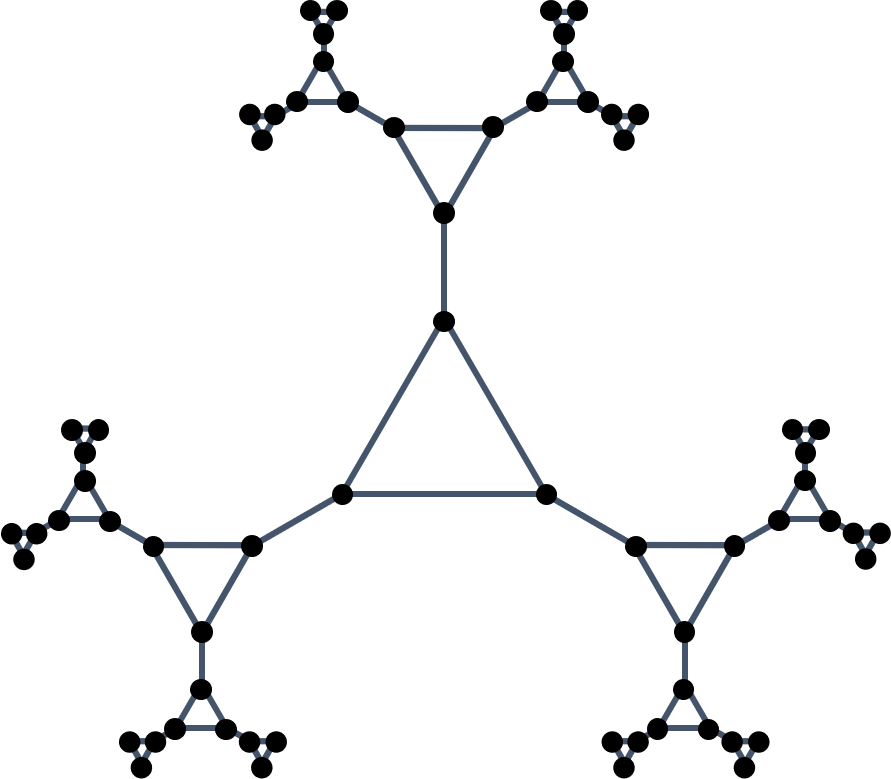}\\[.15in]

  \includegraphics[width=.28\textwidth]{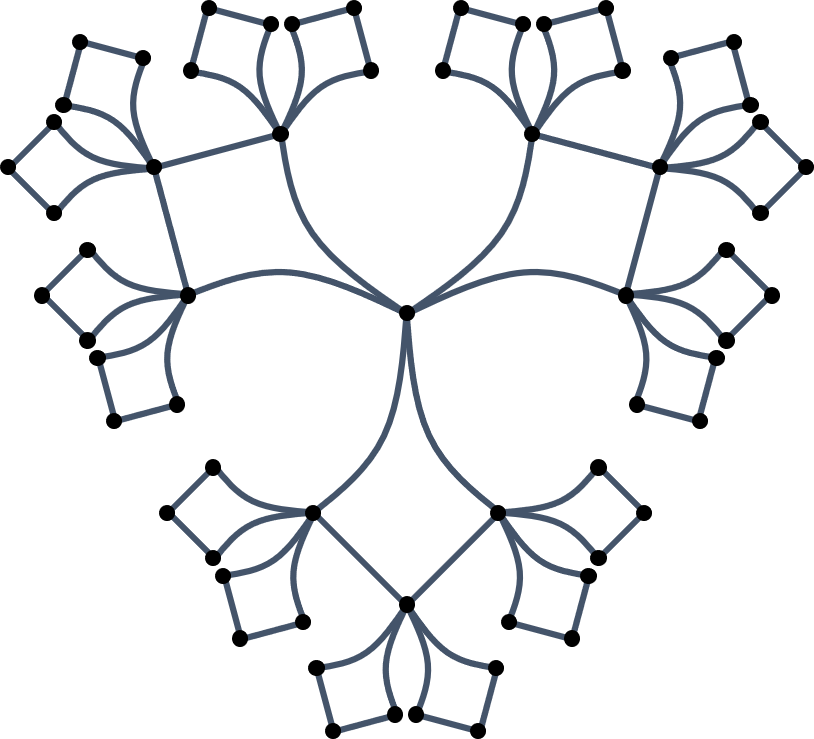}
  \qquad \qquad
  \includegraphics[width=.28\textwidth]{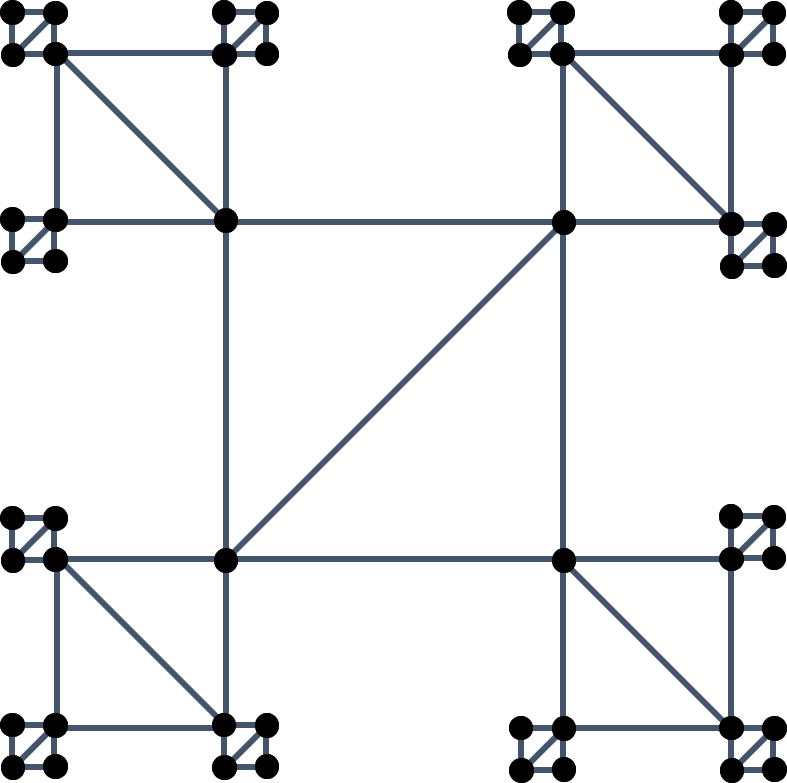}
    \caption{Some infinite graphs (with only a finite portion sketched, repeating in the obvious way).}
  \label{fig:many-products}
\end{figure}
For each such $X$, we are interested in finite graphs~$G$ that ``locally resemble''~$X$. More precisely we are interested in $\FinQuo(X)$, all the finite graphs (up to isomorphism) that are \emph{covered} by~$X$; essentially, this means that there is a graph homomorphism $X \to G$ that is a local isomorphism (see \Cref{def:cover}).  For example, if $X$ is the bottom-left graph in \Cref{fig:many-products}, then $\FinQuo(X)$ consists of the finite graphs that are collections of~$C_4$'s, with each vertex participating in three~$C_4$'s.  Or, if $X$ is the top-right graph in \Cref{fig:many-products}, namely the infinite regular tree~$\ITree_3$, the set $\FinQuo(X)$ is just the $3$-regular finite graphs.

For $\FinQuo(\ITree_3)$, a well known problem is to understand the possible expansion properties of graphs in the set; in particular, to understand the possible eigenvalues of $3$-regular graphs.  Every $3$-regular graph $G$ has a ``trivial'' eigenvalue of~$3$.  As for the next largest eigenvalue, $\lambda_2(G)$, a theorem of Alon and Boppana~\cite{Alo86} implies that for any $\eps > 0$, almost all $G \in \FinQuo(\ITree_3)$ have $\lambda_2(G) \geq 2\sqrt{2} - \eps$.  This special quantity, $2\sqrt{2}$, is precisely the \emph{spectral radius} of the infinite graph~$\ITree_3$.  (In general, the spectral radius of the infinite $\Delta$-regular tree $\ITree_\Delta$ is $2\sqrt{\Delta - 1}$.) The ``Ramanjuan question'' is to ask whether the Alon--Boppana bound is tight; that is, whether there are infinitely many $3$-regular graphs~$G$ with $\lambda_2(G) \leq 2\sqrt{2}$.  Such ``optimal spectral expanders'' are called $3$-regular \emph{Ramanujan graphs}.

In this paper, we investigate the same question for other infinite graphs~$X$, beyond just $\ITree_\Delta$ and other infinite trees.  Take again the bottom-left $X$ in \Cref{fig:many-products}, which happens to be the ``free product graph'' $C_4 \ast C_4 \ast C_4$.  Every $G \in \FinQuo(C_4 \ast C_4 \ast C_4)$ is $6$-regular, so Alon--Boppana implies that almost all these $G$ have $\lambda_2(G) \geq 2\sqrt{5} - \eps$.  But in fact, an extension of Alon--Boppana due to Grigorchuk and \.{Z}uk~\cite{GZ99} shows that almost all $G \in \FinQuo(C_4 \ast C_4 \ast C_4)$ must have $\lambda_2(G) \geq \rho_0 - \eps$, where $\rho_0 =\sqrt{5\sqrt{5}+11} > 2\sqrt{5}$ is the spectral radius of $C_4 \ast C_4 \ast C_4$.  The Ramanujan question for $X = C_4 \ast C_4 \ast C_4$ then becomes: Are there infinitely many \emph{$X$-Ramanujan graphs}, meaning graphs $G \in \FinQuo(C_4 \ast C_4 \ast C_4)$ with $\lambda_2(G) \leq \rho_0$?  We will show the answer is positive.

We may repeat the same question for, say, the sixth graph in \Cref{fig:many-products}, $X = C_3 \ast C_2$, the Cayley graph of the \emph{modular group} $\mathrm{PSL}(2,\Z)$:  We will show there are infinitely many $G \in \FinQuo(C_3 \ast C_2)$ with $\lambda_2(G)$ at most the spectral radius of $C_3 \ast C_2$, namely $\frac12 + \frac12 \sqrt{8\sqrt{2} + 13}$.

Or again, consider the top-left graph in \Cref{fig:many-products}, $X = C_3 \ast C_2 \ast C_2$, a graph determined by Paschke~\cite{Pas93} to have spectral radius $\rho_1 \approx 3.65$, the largest root of \mbox{$x^5 + x^4 - x^3 - 37x^2 - 108 x + 112$}.  Paschke showed this is the infinite graph of smallest spectral radius that is vertex-transitive, $4$-regular, and contains a triangle.  Using Paschke's work, Mohar~\cite[Proof of Theorem 6.2]{Moh10} showed that almost all finite $4$-regular graphs $G$ in which every vertex participates in a triangle must have $\lambda_2(G) \geq \rho_1 - \eps$.  We will show that, conversely, there are infinitely many such graphs~$G$ with $\lambda_2(G) \leq \rho_1$.

\paragraph{Motivations.}  We discuss here some motivations from theoretical computer science and other areas.  The traditional motivation for Ramanujan graphs is their optimal (spectral) expansion property. The usual $\ITree_\Delta$-Ramanujan graphs have either few or no short cycles. However one may conceive of situations requiring graphs with both strong spectral expansion properties and plenty of short cycles. Our $(C_3 \ast C_2 \ast C_2)$-Ramanujan graphs (for example) have this property.  Ramanujan graphs with special additional structures have indeed played a role in areas such as quantum computation~\cite{PS18} and cryptography~\cite{CFLMP18}.  Another recent application comes from a work by Koll\'{a}r--Fitzpatrick--Sarnak--Houck~\cite{KFSH19} on circuit quantum electrodynamics, where finite $3$-regular graphs with carefully controlled eigenvalue intervals (and in particular, $(C_3 \ast C_2)$-Ramanujan graphs) play a key role.

Recently, there has been a lot of interest in high-dimensional expanders (see the survey \cite{Lub17}), which are expander graphs with certain constrained local structure --- for example a 2-dimensional expander is a graph composed of triangles where the neighborhood of every vertex is an expander.  We speculate that the tools introduced in this work might be useful in constructing high dimensional expanders.

Our original motivation for investigating these questions came from the algorithmic theory of random constraint satisfaction problems (CSPs). There are certain predicates (``constraints'') on a small number of Boolean variables that are well modeled by (possibly edge-signed) graphs.  Besides the ``cut'' predicate (modeled by a single edge), examples include: the $\mathsf{NAE}_3$ (Not-All-Equals) predicate on three Boolean variables, which is modeled by the triangle graph~$C_3$; and, the the $\mathsf{SORT}_4$ predicate on four Boolean variables, which is modeled by the graph $C_4$ (with three edges ``negated'').  It is of considerable interest to understand how well efficient algorithms (e.g., eigenvalue/SDP-based algorithms) can solve large CSPs.  The most challenging instances tend to be large \emph{random} CSPs, and this motivates studying the eigenvalues of large random regular graphs composed of, say, triangles (for random $\mathsf{NAE}_3$ CSPs) or $C_4$'s (for random $\mathsf{SORT}_4$).  The methods we introduce in this paper provide a natural way (``additive lifts'') to produce such large random instances, as well as to understand their eigenvalues.  See~\cite{DMOSS19} for more details on the $\mathsf{NAE}_3$ CSP, \cite{Moh18}~for more details on the $\mathsf{SORT}_4$ CSP, and recent followup work of the authors and Paredes~\cite{MOP19} on more general CSPs, including ``Friedman's Theorem''-style results for some random CSPs.

%.  In turn, the quality of the best known algorithms for random CSPs (typically, semidefinite programming algorithms) depends on eigenvalue bounds for
%
%if its vertices are colored by~$0$ and~$1$, then there are either two or zero bichromatic edges, depending on whether the three colors are Not All Equal.  Similarly the ``Sorted'' predicate on four Boolean variables (satisfied if $x_1 \leq x_2 \leq x_3 \leq x_4$ or $x_1 \geq x_2 \geq x_3 \geq x_4$) is well modeled by the graph $C_4$ (with three edges ``negated''); if the vertices are colored by~$0$ and~$1$, then either three or one of the following ``$C_4$-constraints'' are satisfied, depending on whether the four colors are Sorted: $\{x_1 = x_2, x_2 = x_3, x_3 = x_4, x_4 \neq x_1\}$.  It is of considerable interest to understand how well efficient algorithms can solve large CSPs, with the most challenging instances tending to be large \emph{random} CSPs.  The ability XXXQQQ

\section{$\IG$-Ramanujan graphs}
In this section, $\IG = (V,E)$ (and similarly $\IG_1$, $\IG_2$, etc.) will denote a connected undirected graph on at most countably many vertices, with uniformly bounded vertex degrees, and with multiple edges and self-loops allowed.  We also identify $\IG$ with its associated adjacency matrix operator, acting on $\ell^2(V)$.  We recall some basic facts concerning spectral properties of~$\IG$; see, e.g.,~\cite{MW89,Moh82}.
\begin{definition}
    The \emph{spectrum} $\spec(\IG)$ of $\IG$ is the set of all complex $\lambda$ such that $\IG - \lambda \Id$ is not invertible. In fact, $\spec(\IG)$ is a subset of $\R$ since we're assuming $\IG$ is undirected; it is also a compact set.  When $\IG$ is finite, we extend $\spec(\IG)$ to be a multiset --- namely, the roots of the \emph{characteristic polynomial} $\charpoly{\IG}{x}$, taken with multiplicity.  In this case, we write the spectrum (eigenvalues) of~$\IG$ as $\lambda_1(\IG) > \lambda_2(\IG) \geq \cdots \geq \lambda_{|V|}(\IG)$.
\end{definition}

\begin{definition}      \label{def:specrad}
    The \emph{spectral radius} of~$\IG$ is
    \[
        \specrad(\IG) = \limsup_{t \to \infty} \left\{ (c_{uv}^{(t)})^{1/t} \right\},
    \]
    where $c_{uv}^{(t)}$ denotes the number of walks of length~$t$ in~$\IG$ from $u$ vertex $u \in V$ to vertex $v \in V$; it is not hard to show this is independent of the particular choice of $u,v$ \cite[Lemma 1.7]{Woess00}.
\end{definition}
\begin{fact}        \label{fact:specrad-basic}
    The spectral radius $\specrad(\IG)$ is also equal to the operator norm of $\IG$ acting on $\ell^2(V)$, and to $\max\{\lambda : \lambda \in \spec(\IG)\}$.  We also have $\min\{\lambda : \lambda \in \spec(\IG)\} \geq -\specrad(\IG)$.  Finally, $c_{uu}^{(t)} \leq \specrad(\IG)^t$ holds for every $u \in V$ and $t \in \N$.
\end{fact}
\begin{fact}
    If $\IG$ is bipartite then its spectrum is symmetric about~$0$; i.e., $\lambda \in \spec(\IG)$ if and only if $-\lambda \in \spec(\IG)$.  In particular, $-\specrad(\IG) \in \spec(\IG)$.
\end{fact}

\begin{definition} \label{def:cover}
    For graphs $\IG_1$, $\IG_2$, we say that $\IG_1$ is a \emph{cover} %\footnote{Sometimes called a \emph{lift}, but we will reserve this term for later.}
    of $\IG_2$ (and that $\IG_2$ is a \emph{quotient} of $\IG_1$) if there exists $f = (f_V, f_E)$ where $f_V : V_1 \to V_2$ and $f_E : E_1 \to E_2$ are surjections satisfying $f_E(\{u, v\}) = \{f_V(u),f_V(v)\}$, and $f_E$ bijects the edges incident to $u$ in $V_1$ with the edges incident to $f_V(u)$ in $V_2$. The map $f$ is also called a \emph{covering}.
\end{definition}
As an example, one may check that the infinite graph depicted on the left in \Cref{fig:cover-example} is a cover of the finite graph depicted on the right.
\begin{figure}[H]
  \centering
  \includegraphics[width=.3\textwidth]{product-figures/eg1.png} \qquad \qquad \includegraphics[width=.3\textwidth]{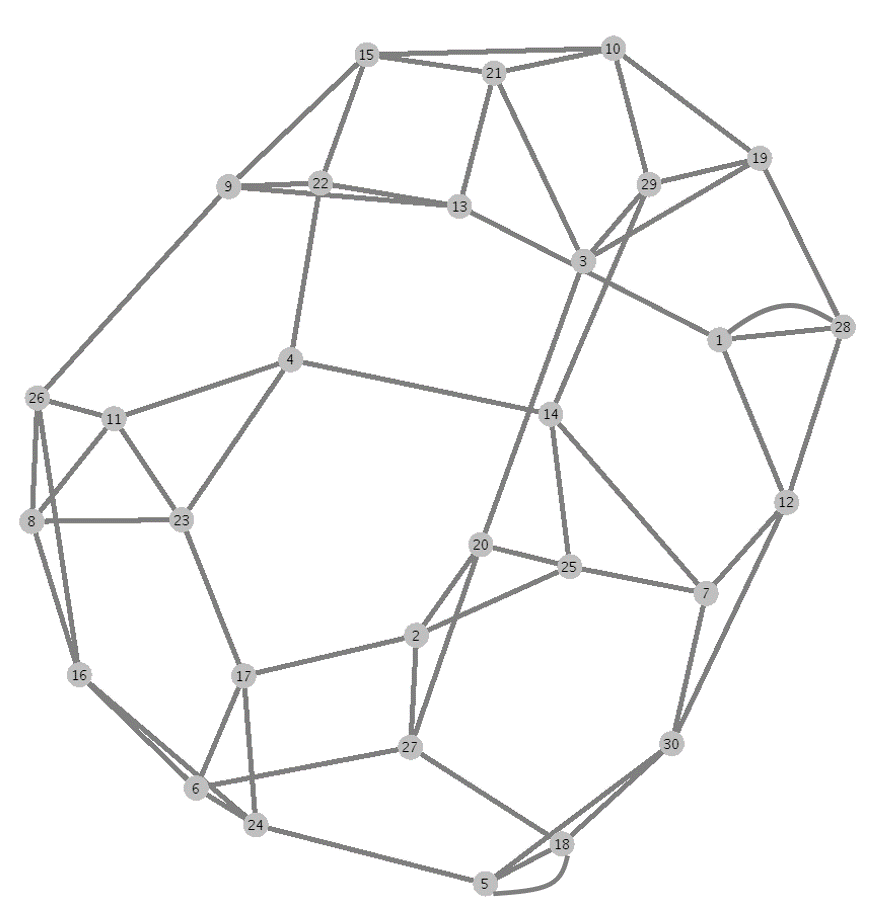}

  \caption{The infinite graph on the left covers the finite graph on the right.}
  \label{fig:cover-example}
\end{figure}
\begin{fact}                    \label{fact:greenberg1}
    If $\IG_1$ is a cover of $\IG_2$, then $\specrad(\IG_1) \leq \specrad(\IG_2)$.  Furthermore, if $\IG_1$ is finite then $\specrad(\IG_1) = \specrad(\IG_2)$.
\end{fact}
The first statement in the above fact is an easy and well known consequence of \Cref{def:specrad}, since distinct closed walks in $\IG_1$ map to distinct closed walks in~$\IG_2$.  The second statement may be considered folklore, and appears in Greenberg's thesis~\cite{Gre95}. We now review some definitions and results from that thesis. (See also~\cite{LN98}.)
\begin{definition}
    Given $\IG$, we write $\FinQuo(\IG)$ to denote the family of finite graphs\footnote{Up to isomorphism.} covered by~$\IG$.  (For this definition, we are generally interested in infinite graphs~$\IG$.) The set $\FinQuo(\IG)$ may be empty; but otherwise,  by a combination of the fact that $G_1$ and $G_2$ have the same universal covering tree, the fact that any two graphs with the same universal covering tree have a common finite cover (due to \cite{Lei82}), and \Cref{fact:greenberg1}, $\specrad(G_1) = \specrad(G_2)$ for all $G_1, G_2 \in \FinQuo(\IG)$.  When \mbox{$\FinQuo(\IG) \neq \emptyset$} we write $\chi(\IG)$ for the common spectral radius of all $G \in \FinQuo(\IG)$.
\end{definition}
\begin{remark}
    If $\IG$ is $\Delta$-regular and $\FinQuo(\IG) \neq \emptyset$ then $\chi(\IG) = \Delta$ (because all $G \in \FinQuo(\IG)$ are $\Delta$-regular).
\end{remark}
We also remark that it is not particularly easy to decide whether or not $\FinQuo(\IG) = \emptyset$, given~$\IG$.  In case $\IG$ is an infinite tree, it is known that $\FinQuo(\IG)$ is nonempty if and only if $\IG$ is the universal cover tree of some finite connected graph~\cite{BK90}.
%, there is a characterization in terms of $\IG$'s automorphism group, due to Bass and Kulkarni~\cite{BK90}.

\subsection{Alon--Boppana-type theorems}
When $\FinQuo(\IG)$ is nonempty, we are interested in the spectrum of graphs $G$ in $\FinQuo(\IG)$.  Each such $G$ will have $\lambda_1(G) = \chi(\IG)$, so we consider the remaining eigenvalues $\lambda_2(G) \geq \cdots \geq \lambda_n(G)$, particularly $\lambda_2(G)$.  In the case that $G$ is $\Delta$-regular, the smallness of  $\lambda_2(G)$ (and $|\lambda_n(G)|$) controls the expansion of~$G$.  If $G$ is bipartite, then $\lambda_n(G) = -\chi(\IG)$ and the bipartite expansion is controlled just by $\lambda_2(G)$.

A notable theorem of Alon and Boppana~\cite{Alo86} is that for all $\Delta \geq 2$ and all $\eps > 0$, there are only finitely many $\Delta$-regular graphs~$G$ with $\lambda_2(G) \leq 2\sqrt{\Delta - 1} - \eps$.  It is well known that $2\sqrt{\Delta - 1}$ is the spectral radius of the infinite $\Delta$-regular tree, $\specrad(\ITree_\Delta)$.  Since $\FinQuo(\ITree_\Delta)$ is the set of all finite connected $\Delta$-regular graphs, the Alon--Boppana Theorem may also be stated as
\[
    \liminf_{n \to \infty} \lambda_2(G_n) \geq \specrad(\ITree_\Delta) \quad \text{for all sequences $(G_n)$ in $\FinQuo(\ITree_\Delta)$.}
\]
(Notice that if we restrict attention to bipartite $\Delta$-regular graphs~$G$ on $n$ vertices, we'll have $\lambda_n(G) = -\Delta$ and we can conclude that there are only finitely many such graphs that have $\lambda_{n-1}(G) \geq -\specrad(\ITree_\Delta) + \eps$.)  Thus the Alon--Boppana gives a limitation on the spectral expansion quality of $\Delta$-regular graphs and bipartite graphs.

There are several known extensions and strengthenings of the Alon--Boppana Theorem.  One strengthening (usually attributed to Serre~\cite{Ser97}) says that for \emph{any} $k \geq 2$, there are only finitely many $\Delta$-regular $G$ with $\lambda_k(G) \leq  \specrad(\ITree_\Delta) - \eps$.  Indeed, in an $n$-vertex $\Delta$-regular graph~$G$, at least $cn$ eigenvalues from $\spec(G_n)$ must be at least $\specrad(\ITree_\Delta) - \eps$, for some $c = c(\Delta,\eps) > 0$.

Another interesting direction, due to Feng and Li~\cite{FL86}, concerns the same questions for $(\Delta_1,\Delta_2)$-biregular graphs.  (This includes the case of $\Delta$-regular bipartite graphs, by taking $\Delta_1 = \Delta_2 = \Delta$.)  These are precisely the graphs $\FinQuo(\IG)$ when $\IG$ is the infinite $(\Delta_1,\Delta_2)$-biregular tree $\ITree_{\Delta_1,\Delta_2}$.  It holds that $\chi(\ITree_{\Delta_1,\Delta_2}) = \sqrt{\Delta_1}\sqrt{\Delta_2}$ and $\specrad(\ITree_{\Delta_1,\Delta_2}) =  \sqrt{\Delta_1 - 1} + \sqrt{\Delta_2 - 1}$.  Feng and Li showed the Alon--Boppana Theorem analogue in this setting,
\[
    \liminf_{n \to \infty} \lambda_2(G_n) \geq \specrad(\ITree_{\Delta_1,\Delta_2}) \quad \text{for all sequences $(G_n)$ in $\FinQuo(\ITree_{\Delta_1,\Delta_2})$;}
\]
they also showed the Serre-style strengthening.  Mohar gave certain generalizations of these results to \emph{multipartite} graphs~\cite{Moh10}.

In these Alon--Boppana(--Serre)-type results for quotients of infinite trees $\IG$, the $\liminf$ lower bound on second eigenvalues for graphs in $\FinQuo(\IG)$ has  been $\specrad(\IG)$.  In fact, it turns out that this phenomenon holds even when $\IG$ is not a tree.  The following theorem was first proved by Greenberg~\cite{Gre95} (except that he considered the absolute values of eigenvalues); see \cite[Chap.~9, Thm.~13]{Li96} and Grigorchuk--\.{Z}uk~\cite{GZ99} for proofs:
\begin{theorem}                                     \label{thm:greenberg}
    For any $\IG$ and any $\eps > 0$, there is a constant $c > 0$ such that for all $n$-vertex $G \in \FinQuo(\IG)$ it holds that at least $cn$ eigenvalues from $\spec(G_n)$ are at least $\specrad(\IG) - \eps$.  (In particular, $\lambda_2(G)$ --- and indeed, $\lambda_k(G)$ for any $k \geq 2$ --- is at least $\specrad(\IG)-\eps$ for all but finitely many $G \in \FinQuo(\IG)$.)
\end{theorem}

\subsection{Ramanujan and $\IG$-Ramanujan graphs}

The original Alon--Boppana Theorem showed that, for any $\eps > 0$, one cannot have infinitely many (distinct) $\Delta$-regular graphs~$G$ with $\lambda_2(G) \leq 2\sqrt{\Delta-1} - \eps$.  But can one have infinitely many $\Delta$-regular graphs~$G$ with $\lambda_2(G) \leq 2\sqrt{\Delta-1}$?  Such graphs are called \emph{($\Delta$-regular) Ramanujan graphs}, and infinite families of them were first constructed (for $\Delta-1$ prime) by Lubotzky--Phillips--Sarnak~\cite{LPS88} and by Margulis~\cite{Mar88}. Let us clarify here the several slightly different definitions of Ramanujan graphs.
\begin{definition}
    We shall call an $n$-vertex, $\Delta$-regular graph $G$ (with $\Delta \geq 3$) \emph{(one-sided) Ramanujan} if $\lambda_2(G) \leq 2\sqrt{\Delta - 1}$, and \emph{two-sided Ramanujan} if in addition $\lambda_n(G) \geq -2\sqrt{\Delta -1}$.  If $G$ is bipartite and (one-sided) Ramanujan, we will call it \emph{bipartite Ramanujan}.  (A bipartite graph cannot be two-sided Ramanujan, as $\lambda_n(G)$ will always be $-\Delta$.)
\end{definition}
The Lubotzky--Phillips--Sarnak--Margulis constructions give % $\Delta$-regular, $n$-vertex two-sided Ramanujan graphs whenever $\Delta-1$ is a prime congruent to~$1$ mod~$4$ and $n = q(q^2-1)/2$ for $q$ a prime congruent to~$1$ mod~$4$ with $\left(\dfrac{\Delta-1}{q}\right) = 1$.
infinitely many $\Delta$-regular two-sided Ramanujan graphs whenever $\Delta-1$ is a prime congruent to~$1$ mod~$4$, and also infinitely many $\Delta$-regular bipartite Ramanujan graphs in the same case.  By 1994, Morgenstern~\cite{Mor94} had extended their methods to obtain the same results only assuming that $\Delta - 1$ is a prime power.  It should be noted that these constructions produced Ramanujan graphs only for certain numbers of vertices~$n$ (namely all $n = p(\Delta)$, for certain fixed polynomials $p$).

To summarize these results: for all $\Delta$ with $\Delta-1$ a prime power, there are infinitely many $G \in \FinQuo(\ITree_{\Delta})$ with $\lambda_2(G) \leq \specrad(\ITree_{\Delta})$.  (We will later discuss the improvement to these results by Marcus, Spielman, and Srivastava \cite{MSS15a,MSS15d}.)  Based on Greenberg's \Cref{thm:greenberg}, it is very natural to ask if analogous results are true not just for other $\Delta$, but for other infinite~$\IG$.  This was apparently first asked for the infinite biregular tree $\ITree_{\Delta_1,\Delta_2}$ by Hashimoto~\cite{Has89} (later also by Li and Sol\'{e}~\cite{LS96}).\footnote{In fact, they defined a notion of Ramanujancy for $(\Delta_1,\Delta_2)$-biregular graphs $G$ that has a flavor even stronger than that of two-sidedness; they required $\spec(G) \setminus \{\pm \chi(\ITree_{\Delta_1,\Delta_2})\} \subseteq \spec(\ITree_{\Delta_1,\Delta_2})$; in other words, for all $\lambda \in \spec(G)$ with $\lambda \neq \pm \sqrt{\Delta_1\Delta_2}$, not only do we have $|\lambda| \leq \sqrt{\Delta_1-1} + \sqrt{\Delta_2-1}$, but also $|\lambda| \geq |\sqrt{\Delta_1-1} - \sqrt{\Delta_2-1}|$.  This definition is related to the Riemann Hypothesis for the zeta function of graphs.} More generally, Grigorchuk and \.Zuk~\cite{GZ99} made (essentially\footnote{Actually, they worked with normalized adjacency matrices rather than unnormalized ones; the definitions are equivalent in the case of regular graphs.  They also only defined one-sided Ramanujancy.}) the following definition:
\begin{definition}  \label{def:x-ramanujan}
    Let $\IG$ be an infinite graph and let $G \in \FinQuo(\IG)$.  We say that $G \in \FinQuo(\IG)$ is \emph{(one-sided) $\IG$-Ramanujan} if $\lambda_2(G) \leq \specrad(\IG)$.  The notions of \emph{two-sided $\IG$-Ramanujan} and \emph{bipartite $\IG$-Ramanujan} (for bipartite~$\IG$) are defined analogously.
\end{definition}
Note that here the notion of Ramanujancy is tied to the infinite covering graph~$\IG$; see Clark~\cite[Sec.~5]{Cla07} for further advocacy of this viewpoint (albeit only for the case that $\IG$ is a tree). Starting with Greenberg and Lubotzky~\cite{Lub94}, a number of authors defined a fixed finite (irregular) graph~$G$ to be ``Ramanujan'' if $\lambda_2(G) \leq \specrad(\UCT(G))$, where $\UCT(G)$ denotes the universal cover tree of~$\IG$~\cite{Ter10}.  However  \Cref{def:x-ramanujan} takes a more general approach, allowing us to ask the usual Ramanujan question:
\begin{question}    \label{ques:inf-x}
    Given an infinite $\IG$, are there infinitely many $\IG$-Ramanujan graphs?
\end{question}

Given $\IG$, an obvious necessary condition for a positive answer to this question is \mbox{$\FinQuo(\IG) \neq \emptyset$}; one should at least have the existence of infinitely many finite graphs covered by~$\IG$.  However this is known to be an insufficient condition, even when $\IG$ is a tree:
\begin{theorem}                                     \label{thm:lubotzky-nagnibeda}
    (Lubotzky--Nagnibeda~\cite{LN98}.)  There exists an infinite tree~$\IG$ such that $\FinQuo(\IG)$ is infinite but contains no $\IG$-Ramanujan graphs.
\end{theorem}
Based on this result, Clark~\cite{Cla06} proposed the following definition and question (though he only discussed the case of $\IG$ being a tree):
\begin{definition}
    An infinite graph~$\IG$ is said to be \emph{Ramanujan} if $\FinQuo(\IG)$ contains infinitely many  $\IG$-Ramanujan graphs.  It is said to be \emph{weakly Ramanujan} if $\FinQuo(\IG)$ contains at least one $\IG$-Ramanujan graph.  (One can apply here the usual additional adjectives two-sided/bipartite.)
\end{definition}
\begin{question}    \label{ques:inf-x-if-weak}
    If $\IG$ is weakly Ramanujan, must it be Ramanujan?
\end{question}
Clark~\cite{Cla06} also made the following philosophical point.  The proof of the Lubotzky--Nagnibeda Theorem only explicitly establishes that there is an infinite tree~$\IG$ with $\FinQuo(\IG)$ infinite but no $G \in \FinQuo(\IG)$ having $\lambda_2(G) \leq \specrad(\IG)$.  But as long as we're excluding $\lambda_1(G)$, why not also exclude  $\lambda_2(G)$, or constantly many exceptional eigenvalues?  Clark suggested the following definition:
\begin{definition}
    An infinite graph~$\IG$ is said to be \emph{$k$-quasi-Ramanujan} ($k \geq 1$) if there are infinitely many $G \in \FinQuo(\IG)$ with $\lambda_{k+1}(G) \leq \specrad(\IG)$.
\end{definition}
Clark observed that a positive answer to the following question (weaker than \Cref{ques:inf-x-if-weak}) is consistent with Greenberg's \Cref{thm:greenberg}:
\begin{question}    \label{ques:inf-x-if-weak-quasi}
    If $\IG$ is weakly Ramanujan, must it be $k$-quasi-Ramanujan for some $k \in \N^+$?
\end{question}

Later, Clark~\cite{Cla07} directly conjectured the following related statement:
\begin{conjecture}  \label{conj:clark}
    (Clark.)  For every finite~$G$, there are infinitely many lifts $H$ of~$G$ such that every ``new'' eigenvalue $\lambda \in \spec(H) \setminus \spec(G)$ has $\lambda \leq \specrad(\UCT(G))$.
\end{conjecture}
(The notion of an \emph{$N$-lift} of a graph will be reviewed in \Cref{sec:add-lift}.) This effectively generalizes an earlier well known conjecture by Bilu and Linial~\cite{BL06}:
\begin{conjecture}
    (Bilu--Linial.)   For every regular finite~$G$, there is a $2$-lift $H$ of~$G$ such that every ``new'' eigenvalue  has $\lambda \leq \specrad(\UCT(G))$.  As a consequence, one can obtain a tower of  infinitely many such lifts of~$G$.
\end{conjecture}
\noindent (Actually, in the above, Bilu--Linial and Clark had the stronger conjecture $|\lambda| \leq \specrad(\UCT(G))$.)  As pointed out in~\cite{Cla07}, one can easily check that the complete bipartite graph $K_{\Delta_1, \Delta_2}$ is $\ITree_{\Delta_1,\Delta_2}$-Ramanujan, and hence a positive resolution of \Cref{conj:clark} would show not only that $\ITree_{\Delta_1,\Delta_2}$ is quasi-Ramanujan, but that there are in fact infinitely many $\ITree_{\Delta_1,\Delta_2}$-Ramanujan graphs.

Finally, Mohar~\cite{Moh10} studied these questions in the case of \emph{multi-}partite trees.  Suppose one has a $t$-partite finite graph where every vertex in the $i$th part has exactly $j$ neighbors in the $j$th part.  Call $D = (d_{ij})_{ij}$ the degree matrix of such a graph.  There are very simple conditions under which a matrix $D = \N^{t \times t}$ is the degree-matrix of \emph{some} $t$-partite graph, and in this case, there is a unique infinite tree $\ITree_D$ with $D$ as its degree matrix.  As an example, the $(\Delta_1,\Delta_2$)-biregular infinite tree corresponds to
\begin{equation}    \label{eqn:bip-D}
    D = \begin{pmatrix} 0 & \Delta_1 \\ \Delta_2 & 0 \end{pmatrix}.
\end{equation}
Mohar conjectured that for all degree matrices~$D$, the answer to \Cref{ques:inf-x-if-weak-quasi} is positive for~$\ITree_D$.  In fact he made a slightly more refined conjecture.  Given~$D$, he defined $k_D = \max \{k : \lambda_k(D) \geq \specrad(\ITree_D)\}$.  (As an example, $k_D = 1$ for the $D$ in \Cref{eqn:bip-D}.)  He then conjectured:
\begin{conjecture}    \label{conj:mohar}
    (Mohar.) For a degree matrix~$D$:
    \begin{enumerate}
        \item \label{item:mohar1} \cite[Conj.~4.2]{Moh10} If $\ITree_D$ is weakly Ramanujan, then it is $k_D$-quasi-Ramanujan.
        \item \label{item:mohar2} \cite[Conj.~4.4]{Moh10} If $k_D = 1$, then $\ITree_D$ is Ramanujan.
    \end{enumerate}
\end{conjecture}
Mohar mentions that \Cref{item:mohar1} might be true simply as the statement ``$\ITree_D$ is $k_D$-quasi-Ramanujan for every~$D$'', but was unwilling to explicitly conjecture this.

\subsection{Interlacing polynomials}
Several of the conjectures mentioned in the previous section were proven by Marcus, Spielman, and Srivastava using the method of ``interlacing polynomials''~\cite{MSS15a,MSS15b,MSS17,MSS15d}.  Particularly, in \cite{MSS15a} they proved Clark's \Cref{conj:clark} and the Bilu--Linial Conjecture (for one-sided/bipartite Ramanujancy). One can also show their work implies \Cref{item:mohar1} of Mohar's \Cref{conj:mohar}.

Their first work~\cite{MSS15a} in particular shows that for all degree~$\Delta$, if one starts with the complete graph $K_{\Delta+1}$ or the complete bipartite graph $K_{\Delta,\Delta}$, one can perform a sequence of $2$-lifts, obtaining larger and larger finite graphs $G_1, G_2, \dots$ each of which has $\lambda_2(G_i) \leq 2\sqrt{\Delta-1}$; i.e., each $G_i$ is one-sided/bipartite $\ITree_{\Delta}$-Ramanujan.

In a subsequent work~\cite{MSS15d}, they showed the existence of $\Delta$-regular bipartite Ramanujan on $n + n$ vertices for each $\Delta$ and each~$n$, as well as $\Delta$-regular one-sided Ramanujan graphs on $2n$ vertices for each~$n$.  The bipartite graphs in this case are $n$-lifts of the (non-simple) $2$-vertex graph with $\Delta$ edges, which one might denote by $\Delta K_2$.  However, Marcus--Spielman--Srivastava don't really analyze them as lifts of $\Delta K_2$ per se.  Rather, they analyze them as sums of $\Delta$ random (bipartition-respecting) permutations of $M_n$, where $M_n$ denotes the perfect matching on $n+n$ vertices.  In general,~\cite{MSS15d} can be thought of as analyzing the eigenvalues of sums of random permutations of a few large graphs (each of which itself might be the union of many disjoint copies of a fixed small graph).  As such, as we will discuss in \Cref{sec:FFC-appendix}, its techniques can be used to construct infinitely many $\IG$-Ramanujan graphs for vertex-transitive free product graphs~$\IG$.

Finally, in a followup work, Hall, Puder, and Sawin~\cite{HPS18} reinterpret~\cite{MSS15d} from the random lift perspective, and generalize it to work for $N$-lifts of any base graph.  Specifically, they prove Clark's \Cref{conj:clark} in a strong way:
\begin{theorem}                                     \label{thm:HPS}
    (Hall--Puder--Sawin.) Let $G$ be a connected finite graph without loops (but possibly with parallel edges).  Then for every $N \in \N^+$, there is an $N$-lift $H$ of $G$ such that every new eigenvalue $\lambda \in \spec(H) \setminus \spec(G)$ has $\lambda \leq \specrad(\UCT(G))$.
\end{theorem}

\subsection{Our work and technical overview}

\begin{theorem}     \label{thm:additive-quasi}
Suppose $X$ is an additive product graph such that $\chi(X)>\specrad(X)$. Then it is $k$-quasi-Ramanujan for some $k\in\N^+$.
\end{theorem}
\noindent (Recall here that $\chi(\IG)$ is  the common spectral radius of all $G \in \FinQuo(\IG)$.)

In order to show that all additive product graphs $X$ are $k$-quasi-Ramanujan for some $k$, we show the existence of an infinite family of finite quotients of $X$, which each have at most $k$ eigenvalues that exceed $\specrad(X)$. To this end, we start with a base graph $H$ that is a quotient of $X$ and show that for each $n\in\N$ there is an $n$-lift $H_n$ of $H$ that is (i) a quotient of $X$, and (ii) at most $|V(H)|$ eigenvalues of $H_n$ exceed $\specrad(X)$.  To show the existence of an appropriate lift, we pick a uniformly random lift from a restricted class of lifts called \emph{additive lifts} that satisfy (i), and show that such a random lift satisfies (ii) with positive probability.

We achieve this with the method of interlacing polynomials, which one should think of as the ``polynomial probabilistic method'', introduced in \cite{MSS15a,MSS15b}. Let's start with a very simple true statement `for any random variable $\bX$, there is a positive probability that $\bX \le \E[\bX']$. For a polynomial valued random variable $\bp$, we could attempt to make the statement `there is positive probability that $\maxroot(\bp)\le\maxroot(\E\bp)$' where $\E[\bp]$ is the polynomial obtained from coefficient-wise expectations of $\bp$. While this statement is not true in general, it is when $\bp$ is drawn from a well structured family of polynomials known as an `interlacing family', described in \Cref{sec:interlacing-families}. We take $\bp$ to be the polynomial whose roots are the new eigenvalues introduced by the random lift, and a key fact we use to get a handle on this polynomial is that it can be seen as the characteristic polynomial of the matrix one obtains by replacing every edge in $\Adj(H)$ with the standard representation of the permutation labeling it. This helps us prove that $\bp$ is indeed drawn from an interlacing family in \Cref{sec:additive-lift-interlacing} and establish bounds on the roots of $\E[\bp]$. We begin by studying the case $n = 2$, where $\bp$ is always the characteristic polynomial of a signing of $\Adj(H)$ and $\E[\bp]$, which we call the \factorpolynomial, generalizes the well-known matching polynomial and equals it whenever $X$ is a tree. We show that the roots of the \factorpolynomial lie in $[-\specrad(X),\specrad(X)]$ by proving a generalization of Godsil's result that the root moments of the matching polynomial of a graph count the number of treelike walks in the graph.  In particular, we prove that the $t$-th root moment of the \factorpolynomial counts the number of length-$t$ `freelike walks' in $H$, which are defined in Section \Cref{sec:freelike} and obtain the required root bound by upper bounding the number of freelike walks.  Our final ingredient is showing that the same root bounds apply to $\E[\bp]$ for general $n$, which we show follows from the case when $n=2$ in Section \Cref{sec:minors-root-bounds} using representation theoretic machinery.

The main avenue where our techniques differ from those in previous work (i.e. \cite{MSS15a} and \cite{HPS18}) are in how the root bounds are proven on the expected characteristic polynomial.  The proof that the root moments of the matching polynomial count closed treelike walks from \cite{God81} exploits combinatorial structure of the matching polynomial that is not shared by the additive characteristic polynomial.  Instead, we prove the analogous statement to Godsil's result that we need by relating certain combinatorial objects resembling matchings with a particular kind of walk using Newton's identities and Viennot's theory of heaps.

\section{Additive Products and Additive Lifts} \label{sec:lifts}

\subsection{Elementary definitions}
\begin{definition}
    Let $A$ be an $n \times n$ matrix with entries in a commutative ring.  We identify~$A$ with a directed, weighted graph on vertex set~$[n]$ (with self-loops allowed, but no parallel arcs); arc $(i,j)$ is present if and only if $A[i,j] \neq 0$.  In the \emph{general case}, the entries $A[i,j]$ are taken to be distinct formal variables. In the \emph{unweighted} case, each entry $A[i,j]$ is either~$0$ or~$1$; in this case, $A$~is the \emph{adjacency matrix} of the underlying directed graph.  The \emph{plain} case is defined to be when $A$~is unweighted, symmetric, and with diagonal entries~$0$; in this case, $A$ is the adjacency matrix of a simple undirected graph (with each undirected edge considered to be two opposing directed arcs).
\end{definition}

\begin{definition}
    Throughout this work we will consider finite sequences $A_1, \dots, A_\atms$ of matrices over the same set of vertices~$[n]$; we call each $A_j$ an \emph{\atom}, and the index~$j$ its \emph{color}.  We use the terms \emph{general} / \emph{unweighted} / \emph{plain} whenever all $A_j$'s have the associated property; we will also use the term \emph{monochromatic} when $\atms = 1$. When the $A_j$'s are thought of as graphs, we call $G = A_1 + \cdots + A_\atms$ the associated \emph{sum graph}; note that even if all the $A_j$'s are unweighted, $G$ may not be (it may have parallel edges). In the plain case, we use the notation $uCv$ for $u,v\in V$ and $C\in [c]$ to denote an edge $\{u,v\}$ that occurs in $A_C$.
\end{definition}
\begin{definition}      \label{def:edge-atoms-plain}
    A common sum graph case will be when~$G$ is a simple undirected graph with~$\atms$ (undirected) edges, and $A_1, \dots, A_\atms$ are the associated single-edge graphs on $G$'s vertex set~$[n]$; we call $A_1, \dots, A_\atms$ the \emph{edge \atoms} for~$G$. Note that this is an instance of the \emph{plain} case.
\end{definition}

\subsection{The additive product}

In this section we introduce the definition of the \emph{\additiveproduct} of \atoms.  This is a ``quasi-transitive'' infinite graph, meaning one whose automorphism group has only finitely many orbits.  For simplicity, we work in the plain, connected case.  

\begin{definition}      \label{def:additive-product}
    Let $A_1, \dots, A_\atms$ be plain \atoms on common vertex set~$[n]$.  Assume that the sum graph $G = A_1 + \cdots + A_\atms$ is connected; letting $\ul{A}_j$ denote $A_j$ with isolated vertices removed, we also assume that each $\ul{A}_j$ is nonempty and connected.  We now define the (typically infinite) \emph{\additiveproduct} graph $A_1 \addprod \cdots \addprod A_\atms \coloneqq (V, E)$ where $V$ and $E$ are constructed as follows.

    Let $v_1$ be a fixed vertex in $[n]$; let $V$ be the set of strings of the form $v_1 C_1 v_2 C_2 \cdots v_k C_k v_{k+1}$ for $k\ge 0$ such that:
    \begin{enumerate}[label=\emph{(\roman*)}]
    \item each $v_i$ is in $[n]$ and each $C_i$ is in $[c]$,
    \item $C_i \ne C_{i+1}$ for all $i < k$,
    \item $v_i$ and $v_{i+1}$ are both in $\ul{A}_{C_i}$ for all $i \le k$;
    \end{enumerate}
    and, let $E$ be the set of edges on vertex set $V$ such that for each string $s\in V$,
    \begin{enumerate}[label=\emph{(\roman*)}]
    \item we let $\{sCu,sCv\}$ be in $E$ if $\{u,v\}$ is an edge in $\ul{A}_{C}$,
    \item we let $\{sCu,sCuC'v\}$ be in $E$ if $\{u,v\}$ is an edge in $\ul{A}_{C'}$, and
    \item we let $\{v_1,v_1Cv\}$ be in $E$ if $\{v_1,v\}$ is an edge in $\ul{A}_{C}$.
    \end{enumerate}
\end{definition}
\noindent Two examples are given in \Cref{fig:add-prod-eg}.
\begin{figure}[H]
  \centering
  \vspace{-.25in}
  \includegraphics[width=.9\textwidth]{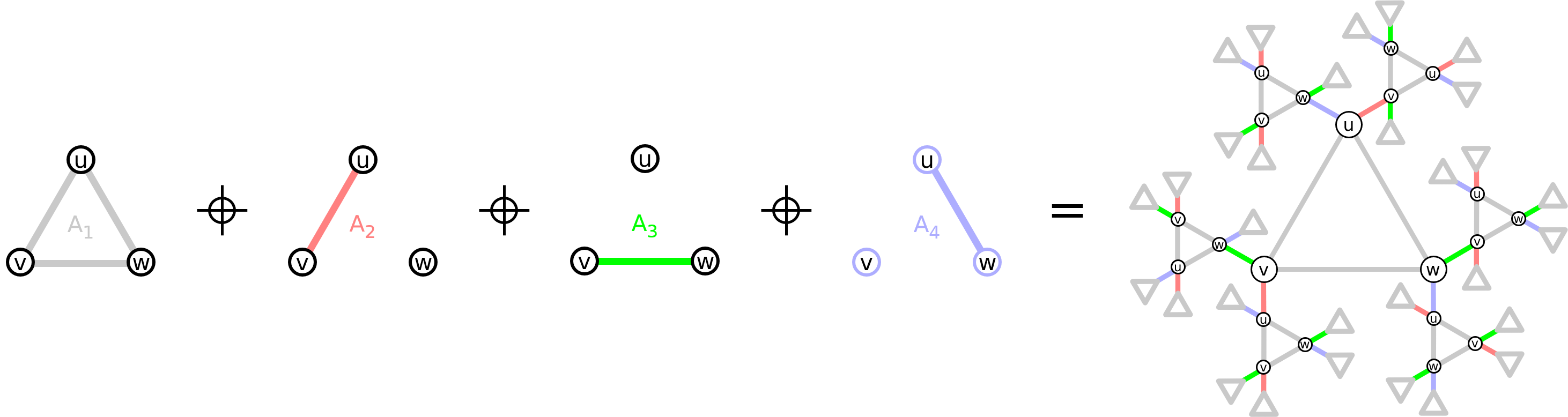} \\
  \vspace{.25in}
  \includegraphics[width=.6\textwidth]{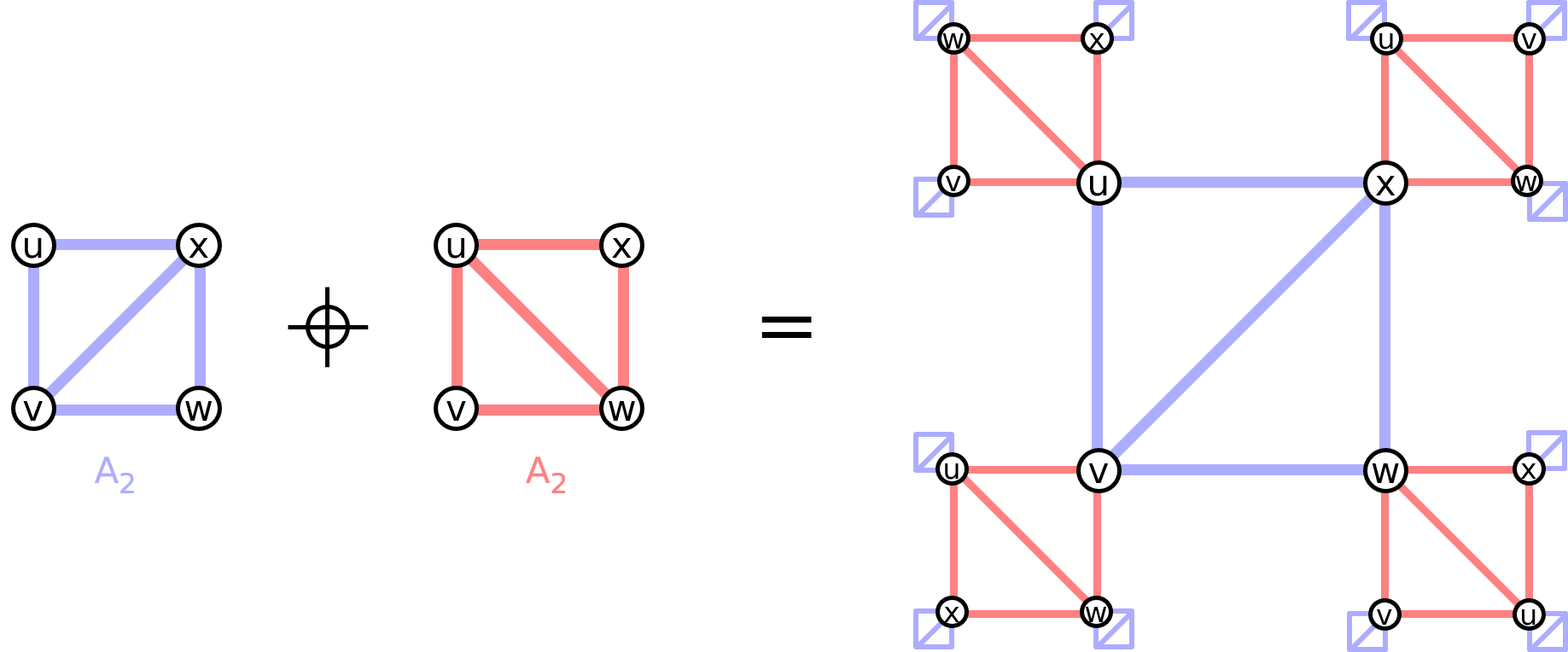}

  \caption{Two examples of the additive product; of course, only a part of each infinite graph can be shown.}
  \label{fig:add-prod-eg}
\end{figure}
\begin{proposition}
    $A_1 \addprod \cdots \addprod A_\atms$ is well defined up to graph isomorphism, independent of choice of~$v_1$.
\end{proposition}
\begin{proof}
    Let $X$ be the additive product graph generated by selecting $v_1$ to be some $u$ and let $X'$ be the additive product graph generated by selecting $v_1$ to be $u'\ne u$. To establish the proposition, we will show that $X$ and $X'$ are isomorphic. First define
    \begin{align*}
        f(u_1 C_1 &\dots C_k u_{k+1}, u_{k+1}C_{k+1}\dots C_{k+\ell}u_{k+\ell + 1}) =\\
        &\begin{cases}
            f(u_1 C_1 \dots C_{k-1} u_k, u_{k+2}\dots C_{k+\ell}u_{k+\ell+1}) & \text{if $C_k = C_{k+1}$ and $u_k = u_{k+2}$,}\\
            u_1 C_1 \dots u_k C_k u_{k+2}C_{k+2}\dots C_{k+\ell}u_{k+\ell+1} &\text{if $C_k = C_{k+1}$ and $u_k \ne u_{k+2}$}\\
            u_1 C_1 \dots u_k C_k u_{k+1} C_{k+1} \dots C_{k+\ell} u_{k+\ell+1} & \text{otherwise.}
        \end{cases}
    \end{align*}

    Let $P = u'C_1u_1\dots C_k u_k C_{k+1} u$ be a string such that $u_i$ and $u_{i+1}$ are both in $\ul{A}_{C_{i+1}}$ (identifying $u'$ with $u_0$ and $u$ with $u_{k+1}$).  We claim that $\Phi(s) \coloneqq f(P, s)$ that maps $V(X)$ to $V(X')$ is an isomorphism.  Define $P'$ as the string obtained by reversing $P$, i.e., define $P'$ as $uC_{k+1}u_kC_k\dots u_1C_1u'$.  To see $\Phi(s)$ is a bijection, consider the map $\Phi'(s)\coloneqq f(P',s)$ that maps $s\in V(X')$ to $V(X)$.  It can be verified that $\Phi'$ is the inverse of $\Phi$, and thus $\Phi$ is bijective.  It can also be verified that if $s$ and $s'$ in $V(X)$ share an edge, then so do $\Phi(s)$ and $\Phi(s')$.
\end{proof}

\begin{proposition}  \label{prop:addprod-covers-sum-graph}
    $A_1\addprod\cdots\addprod A_\atms$ covers $G = A_1 + \cdots + A_\atms$.
\end{proposition}
\begin{proof}
    Define $f_V(u_1C_1\dots C_k u_{k+1}) = u_{k+1}$. For any edge $\{s,t\}$ in $A_1\addprod\cdots \addprod A_\atms$, without loss of generality assume that the string corresponding to $s$ is at least as long as the string corresponding to $t$. And now define $f_E(\{s, t\}) := f_V(s)Cf_V(t)$ where $C$ is the last color that appears in $s$. It can be verified that $(f_V, f_E)$ is a valid covering map.
\end{proof}

Now, we will go over some common infinite graphs and see how they are realized as additive products.

\begin{fact}    \label{fact:uct}
    When $G$ is a connected $\atms$-edge graph with edge \atoms $A_1, \dots, A_\atms$, the \additiveproduct $A_1 \addprod \cdots \addprod A_\atms$ coincides with the \emph{universal cover tree} of~$G$.
\end{fact}
\begin{proof}
    Indeed, the additive product of edge \atoms is a tree, which by \Cref{prop:addprod-covers-sum-graph} covers $G$. It coincides with the universal cover tree since all trees that cover $G$ are isomorphic. %(See \cite[Section 2.3]{MSS15a} for a discussion surrounding the universal cover tree.)
\end{proof}

\begin{fact}    \label{fact:free-prod}
    When each \atom $A_j$ is a (nonempty) vertex-transitive graph on vertex set~$[n]$, the \additiveproduct $A_1 \addprod \cdots \addprod A_\atms$ coincides with the \emph{free product} $A_1 \ast \cdots \ast A_\atms$, as defined for vertex-transitive graphs by Zno\u{\i}ko~\cite{Zno75}, and for general rooted graphs by Quenell~\cite{Que94}.
\end{fact}

\begin{fact} \label{fact:any-free}
    In fact, given vertex-transitive graphs $G_1,\dots,G_\atms$ on vertex sets of possibly different sizes (e.g., Cayley graphs of finite groups), we can also realize their free product as an additive product, as follows. Let $m = \mathsf{LCM}(|V(G_1)|, \dots, |V(G_\atms)|)$. Consider the following atom graphs on $m$ vertices: For each $i$ and $0\le t < m/|V(G_i)|$ let $A_{i,t}$ be a copy of $G_i$ placed on vertices $\{t|V(G_i)|+1, \dots, (t+1)|V(G_i)|\}$. Then the graph
    \[
        A_{1,0} \addprod \cdots \addprod A_{1,m/|V(G_1)|-1} \addprod A_{2,0} \addprod \cdots \addprod A_{2,m/|V(G_2)|-1} \addprod \cdots \addprod A_{\atms, 0} \addprod \cdots \addprod A_{\atms, m/|V(G_\atms)|-1}
    \]
    is isomorphic to the free product $G_1 \ast \cdots \ast G_\atms$.
\end{fact}
\Cref{fig:free-prod-eg} illustrates \Cref{fact:any-free} in the case of the free product $C_3 \ast C_2$ (the Cayley graph of the modular group, the sixth graph in \Cref{fig:many-products}).
\begin{figure}[H]
  \centering
  \includegraphics[width=.6\textwidth]{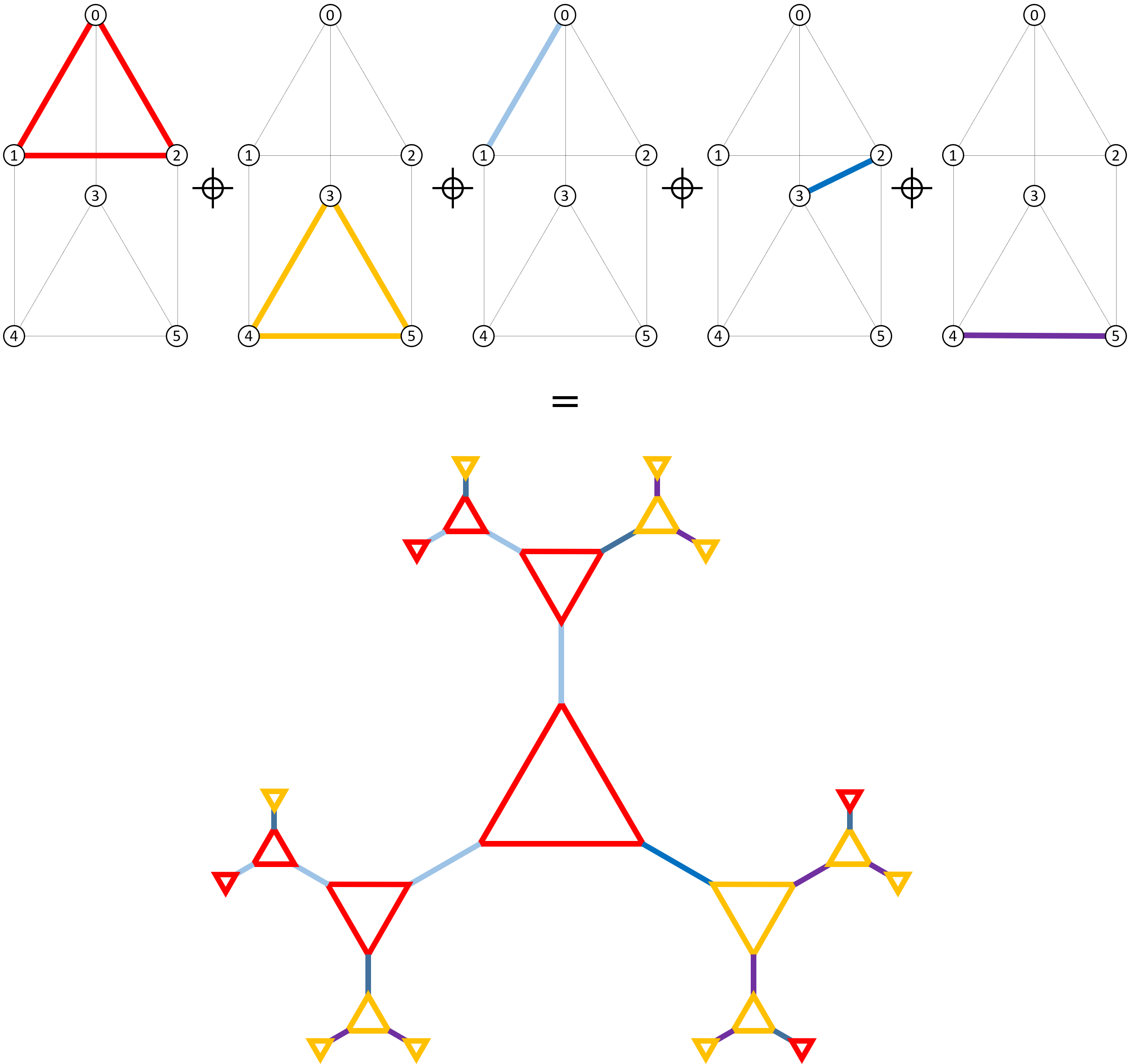}
    \caption{Realizing $C_3 \ast C_2$ as an additive product of five atoms.}
  \label{fig:free-prod-eg}
\end{figure}

\begin{examples}
    All of the graphs in \Cref{fig:many-products} are additive products.  The first and last are from \Cref{fig:add-prod-eg}.  The second through fourth (namely, $\ITree_3$, $\ITree_{4}$, and the biregular $\ITree_{3,4}$) are all universal cover trees, and hence additive products by \Cref{fact:uct}.  The fifth and seventh are additive products as from \Cref{fact:free-prod}; they are $C_4 \addprod C_4$ and $C_4 \addprod C_4 \addprod C_4$, respectively.  Finally, the sixth graph is an additive product as illustrated in \Cref{fig:free-prod-eg}.
\end{examples}

\begin{remark}
    Let $X$ be the second additive product graph in \Cref{fig:add-prod-eg}. Our techniques show the existence of an infinite family of $X$-Ramanujan graphs --- $X$ is a notable example of a graph that is neither a tree nor a free product of vertex transitive graphs that is also an additive product.
\end{remark}

\subsection{Lifts and balanced lifts} \label{sec:add-lift}

\begin{definition}
    In this section, a \emph{graph} $G = (V,E)$ will mean a (possibly infinite) undirected graph, with parallel edges allowed, but loops disallowed.  Thus $E$ should be thought of as a multiset (with its elements being sets of cardinality~$2$).  %All our infinite graphs will have countably many vertices and uniformly bounded degree.
\end{definition}
\begin{definition}
    Given an (undirected) graph $G = (V,E)$, its \emph{directed version} is the directed graph $\vec{G} = (V,\vec{E})$, where the multiset $\vec{E}$ is formed replacing each edge $\{u,v\} \in E$ with a corresponding \emph{dart} --- i.e., pair of directed edges $(u,v)$, $(v,u)$.  Given one edge~$e$ in such a dart, we write $e^{-1}$ for the other edge.  A warning: if $E$ has $c$ copies of an edge $\{u,v\}$, then $\vec{E}$ will contain $c$ pairs $(u,v)$, $(v,u)$, and the $(\cdot)^{-1}$ notation refers to a fixed perfect matching on those pairs.
\end{definition}
%\begin{definition}
%    In a directed graph, a \emph{closed walk} refers to a sequence of edges $(u_0, v_0)$, $(u_1, v_1)$, \dots, $(u_{k-1},v_{k-1})$ such that $v_i = u_{i+1 \text{ mod } k}$ for all~$0 \leq i < k$.
%\end{definition}

\begin{definition}
    Given an $\vtcs$-vertex graph $G = (V,E)$, we identify the vertices with an orthonormal basis of $\C^\vtcs$; the vector for vertex~$v$ is denoted $\ket{v}$.\footnote{We are using the Dirac bra-ket notation, in which $\ket{v}$ denotes a column vector and $\bra{v}$ denotes its conjugate-transpose $\ket{v}^\dagger$.}  A directed edge $(u,v) \in \vec{E}$ may be associated with the matrix $\ketbra{v}{u}$. The \emph{adjacency matrix} of $G$ is the Hermitian matrix $\Adj(G) = \C^{\vtcs \times \vtcs}$  defined by
    \[
        \Adj(G) = \sum_{(u,v) \in \vec{E}} \ketbra{v}{u}.
    \]
\end{definition}
\begin{definition}  \label{def:lift}
    Let $G = (V,E)$ be a graph and let $\sigma : \vec{E} \to \symm{\nlift}$ be a labeling of directed edges by permutations of~$[\nlift]$ satisfying $\sigma(e^{-1}) = \sigma(e)^{-1}$.  The associated \emph{$\nlift$-lift} graph is defined as follows.  The lifted vertex set is $V \times [\nlift]$; we may sometimes identify these vertices with vectors $\ketket{v}{i} \in \C^V \otimes \C^\nlift$. The lifted oriented edge set consists of a dart $((u,i), (v,\sigma(u,v)i))^\pm$ for each dart $(u,v)^\pm \in \vec{E}$ and each $i \in [\nlift]$.
\end{definition}
\begin{definition}
    Given an $\vtcs$-vertex graph $G = (V,E)$ and a $\repdim \in \N^+$ we introduce the \emph{$\repdim$-extended adjacency matrix} $\Adj_\repdim(G)$, a Hermitian matrix in $\C^{\vtcs \repdim \times \vtcs \repdim}$ defined by
    \[
        \Adj_\repdim(G) = \Adj(G) \otimes \Id_\repdim,
    \]
    where $\Id_\repdim$ denotes the $\repdim \times \repdim$ identity operator.  When $\repdim = 1$ this is the usual adjacency matrix. In general, $\Adj_\repdim(G)$ may be thought of as the adjacency matrix of the \emph{trivial} $\repdim$-lift of~$G$, the one where all directed edges are labeled by the identity permutation.  This graph consists of $\repdim$ disjoint copies of~$G$.
\end{definition}

\begin{definition}
    Given a graph $G = (V,E)$ and a group $\Gamma$, a \emph{$\Gamma$-potential} is simply an element~$Q$ of the direct product group $\Gamma^V$.  We think of $Q$ as assigning a group element, written $Q_v$ or $Q(v)$, to each vertex of~$G$.  We will be concerned almost exclusively with the case $\Gamma = \symm{\nlift}$, the symmetric group.
\end{definition}

Let $\Unitaries{\repdim}$ denote the $\repdim$-dimensional unitary matrices.  In this work we will assume that all group representations are unitary.
\begin{definition}
    Recall that if  $\Gamma$ is a group with $\repdim$-dimensional representation $\pi : \Gamma \to \Unitaries{\repdim}$, and $V$ is a set of cardinality~$\vtcs$, then the associated \emph{outer tensor product representation} of the group $\Gamma^V$ is $\pi^{\boxtimes V} : \Gamma^V \to \Unitaries{\vtcs \repdim}$ defined by
    \[
        \pi^{\boxtimes V}(Q) = \sum_{v \in V} \ketbra{v}{v} \otimes \pi(Q_v).
    \]
\end{definition}

\begin{definition}  \label{def:liftAdj}
    Given an $\vtcs$-vertex graph~$G = (V,E)$, a $\Gamma$-potential~$Q$, and a $\repdim$-dimensional unitary representation $\pi : \Gamma \to \Unitaries{\repdim}$, we introduce the notation $\LiftAdj{Q}{\pi}(G)$ for the $\vtcs \repdim$-dimensional \emph{$\pi$-lifted adjacency matrix}
    \begin{align*}
        \LiftAdj{Q}{\pi}(G) &= \pi^{\boxtimes V}(Q)^\dagger \cdot \Adj_\repdim(G) \cdot \pi^{\boxtimes V}(Q)
                                     = \sum_{(u,v) \in \vec{E}} \ketbra{v}{u} \otimes \pi(Q_v^{-1} Q_u).
    \end{align*}
\end{definition}
\begin{remark}  \label{rem:more-general}
    In the setting of \Cref{def:liftAdj}, consider the directed edge-labeling $\sigma : \vec{E} \to \Gamma$ defined by $\sigma(u,v) = Q_v^{-1} Q_u$.  Then $\LiftAdj{Q}{\pi}$ is the matrix
    \[
        \sum_{(u,v) \in \vec{E}} \ketbra{v}{u} \otimes \pi(\sigma(u,v)).
    \]
    This matrix was introduced by Hall, Puder, and Sawin~\cite{HPS18} under the notation $A_{\sigma, \pi}$.  In fact, they studied such matrices for \emph{general} edge-labelings $\sigma$ satisfying $\sigma(e^{-1}) = \sigma(e)^{-1}$, not just the so-called \emph{balanced} ones arising as $\sigma(u,v) = Q_v^{-1} Q_u$ from a potential~$Q$.  We will recover their level of generality shortly, when we consider sum graphs.
\end{remark}
\begin{remark}
    For many representations $\pi$ --- e.g., the \emph{standard} representation \mbox{$\stdrep : \symm{\nlift} \to \Unitaries{\nlift-1}$} of the symmetric group --- it is not natural to pick one particular unitary representation among all the isomorphic ones.  However if $\pi_1$ is isomorphic to~$\pi_2$ via the unitary~$U$, then $\pi_1^{\boxtimes V}$ is conjugate to~$\pi_2^{\boxtimes V}$ via the unitary $\Id \otimes U$, and the same is true of $\LiftAdj{Q}{\pi_1}(G)$ and $\LiftAdj{Q}{\pi_2}(G)$.  Hence these two matrices have the same spectrum and characteristic polynomial, which is what we mainly care to study anyway.
\end{remark}

\begin{remark}
    Another representation of $\symm{\nlift}$ is the $1$-dimensional \emph{sign} representation, $\sgnrep : \symm{\nlift} \to \{\pm 1\}$.  When $\nlift = 2$, the $\stdrep$ and $\sgnrep$ representations coincide, and we obtain the well-known correspondence between $2$-lifts and \emph{edge-signings} of the adjacency matrix of~$G$.
\end{remark}

\begin{definition}      \label{def:balanced-lift}
    Suppose $G = (V,E)$ is a graph.  When $\Gamma = \symm{\nlift}$ and $\pi : \symm{\nlift} \to \Unitaries{\nlift}$ is the usual \emph{permutation} representation, the matrix $\LiftAdj{Q}{\permrep}(G)$ is the adjacency matrix of a certain $\nlift$-lift of~$G$, which we call a \emph{balanced $\nlift$-lift}.  We write~$G^Q$ for this lifted graph, the one obtained from the edge-labeling $\sigma(u,v) = Q_v^{-1} Q_u$ discussed in \Cref{rem:more-general}.  The vertex set of $G^Q$ is $V \times [\nlift]$, and the directed edge set is formed as follows: for each dart $e = (u,v)^{\pm} \in \vec{E}$ and each $i \in [\nlift]$, we include dart $((u,i),(v,j))^{\pm}$ if and only if $Q_v j = Q_u i$.  We remark that a balanced $\nlift$-lift $G^Q$ conists of $\nlift$ disjoint copies of~$G$.
\end{definition}
\begin{remark}  \label{rem:acyclic}
    In general, not all lifts of~$G$ are balanced lifts.  This \emph{is} true, though, if $G$ is an acyclic graph; indeed, it's not hard to check that for each connected component of an acyclic graph, the balanced lifts are in $|\Gamma|$-to-$1$ correspondence with general lifts.  As mentioned earlier, we will recover the full generality of lifts shortly when we consider sum graphs.
\end{remark}

\subsection{Additive lifts}     \label{sec:additive-lifts}

\begin{definition}   \label{def:additive-lift}
    When $H = A_1 + \cdots + A_\atms$ is a sum graph on vertex set~$V$, $\calQ = (Q_1, \dots, Q_{\atms})$ is a sequence of $\Gamma$-potentials $Q_i : V \to \Gamma$, and $\pi : \Gamma \to \Unitaries{\repdim}$ is a representation, we introduce the notation
    \[
        \LiftAdj{\calQ}{\pi}(H)= \sum_{i=1}^\atms \LiftAdj{Q_i}{\pi}(A_i).
    \]
    In case $\Gamma = \symm{\nlift}$ and $\pi = \permrep$, this is the adjacency matrix of a new sum graph
    \[
        H^{\calQ} = A_1^{Q_1} + \cdots + A_\atms^{Q_\atms}
    \]
    that we call the \emph{additive $\nlift$-lift of~$H$ by~$\calQ$}.  Here a balanced $\nlift$-lift is performed on each \atom, and the results are summed together.
\end{definition}
\begin{remark}  \label{rem:additive-general-lift}
    Suppose we regard an ordinary graph $G = (V,E)$ as a sum graph $A_1 + \cdots + A_\atms$ where the \atoms $A_j$ are single edges, as in \Cref{def:edge-atoms-plain}.  In this case, every $\nlift$-lift of $A_j$ is a balanced $\nlift$-lift (indeed, in $|\Gamma|$ different ways, as noted in \Cref{rem:acyclic}).  Thus the  additive $\nlift$-lifts of~$G$ --- when viewed again as ordinary graphs --- recover all ordinary $\nlift$-lifts of~$G$.
\end{remark}

\begin{remark}      \label{rem:spectrum-of-lift}
    As is well known, the spectrum of $H^\calQ$ --- i.e., of $\LiftAdj{\calQ}{\permrep}(H)$ --- is the multiset-union of the ``old'' spectrum of $\Adj(H)$ (of cardinality $\vtcs$), as well as ``new'' spectrum (of cardinality $\vtcs \nlift - \vtcs$).  As observed in~\cite{HPS18}, this ``new'' spectrum is precisely the spectrum of $\LiftAdj{\calQ}{\stdrep}(H)$, where $\stdrep : \symm{\nlift} \to \Unitaries{\nlift-1}$ is the \emph{standard} representation of~$\symm{\nlift}$.
\end{remark}

The following fact is important for the proof of our main theorem.
\begin{proposition}    \label{prop:lift-is-quotient}
    If $H^\calQ$ is connected, then it is a quotient of $X = A_1\addprod \cdots \addprod A_\atms$.
\end{proposition}
\begin{proof}
    Recall that $H^\calQ$ can be written as $A_1^{Q_1}+\cdots + A_\atms^{Q_\atms}$. Expressing each $A_i^{Q_i}$ as a sum graph of $\nlift$ disjoint copies of $A_i$, we obtain an expression of $H^\calQ$ as a sum graph of $\nlift\atms$ \atoms.
    \[
        H^\calQ = \sum_{i=1}^c \sum_{j=1}^\nlift A_{i,j}
    \]
    We show that the additive product $X' = A_{1,1} \addprod \cdots \addprod A_{1,\nlift} \addprod \cdots \cdots \addprod A_{c,1} \addprod \cdots \addprod A_{c,\nlift}$ is isomorphic to $X$, and our proposition then follows from \Cref{prop:addprod-covers-sum-graph}.

    We note the following:
    \begin{enumerate}
        \item \label{item:unique-lift-map} For any $u,v$ in $H$ such that $u,v\in\ul{A}_C$ and $j\in [\nlift]$, there is unique $C',j\in[\nlift]$ such that $(u,j)$ and $(v,j')$ are in $\ul{A}_{C,C'}$.
        \item \label{item:unique-lift-edge} Further, for any edge $uCv$ in $H$ and $j\in[\nlift]$, there is unique $C',j'\in[\nlift]$ such that $(u,j)(C,C')(v,j')$ is an edge.
        \item \label{item:valid-projection} For any pair $(u,j),(v,j')$ in $\ul{A}_{C,C'}$ in $H^\calQ$, $u$ and $v$ must be in $\ul{A}_C$.
        \item \label{item:valid-projection-edge} For any edge $(u,j)(C,C')(v,j')$ in $H^\calQ$, $uCv$ must be an edge in $H$.
    \end{enumerate}

    Without loss of generality, we can assume $X$ is rooted at $u$ and $X'$ is rooted at $(u,1)$. For any vertex $v = (u,1)(C_1,C_1')(u_2,i_2)(C_2,C_2')\dots(C_k,C_k')(u_{k+1},i_{k+1})$ in $X'$ where each $(C_i,C_i')\in [c]\times [\nlift]$, define $\vphi(v)$ as $uC_1u_2\dots C_ku_{k+1}$. $\vphi$ can be seen as an isomorphism from observations \Cref{item:unique-lift-map}, \Cref{item:unique-lift-edge}, \Cref{item:valid-projection}, and \Cref{item:valid-projection-edge}.
\end{proof}

\begin{remark}  \label{rem:connected-comp-lift-is-quotient}
    Observe that even when $H^{\calQ}$ is not connected, the proof of \Cref{prop:lift-is-quotient} gives us that every connected component of $H^{\calQ}$ is a quotient of $A_1\addprod\cdots\addprod A_\atms$. In fact, it tells us that for any connected component of $H^\calQ$ composed of atoms $B_1+\cdots+B_t$, the additive product $B_1\addprod \cdots \addprod B_t$ is isomorphic to $A_1 \addprod \cdots \addprod A_\atms$.
\end{remark}

\section{The \factorpolynomial} \label{sec:factor-poly}

In this section we introduce the \emph{additive characteristic polynomial} of a sequence of matrices.  This generalizes the matching polynomial of a graph. It also generalizes the ``$r$-characteristic polynomial'' of a matrix introduced in~\cite{Rav16,LR18}; see \Cref{rem:mohan}.

\begin{definition}
    Given a ``vertex set'' $[n]$, a \emph{walk} $\omega$ is a sequence of vertices $v_0 v_1 \dots v_t$. We also represent a walk as a corresponding sequence of directed edges $(v_i, v_{i+1})$ between consecutive vertices. We call the walk \emph{closed} if $v_0= v_t$, and call $t$ the \emph{length} of the walk. A self-avoiding walk is a walk $v_0\dots v_t$ where all vertices are distinct with the exception that we allow $v_0=v_t$.  Given a matrix $A$ indexed by~$[n]$, the associated \emph{weight} of walk~$\omega$ is $w_A(\omega) = \prod_{(i,j)\in\omega} A_{ij}$.  In the unweighted case, this is $0$ or~$1$ depending on whether or not $\omega$ is a valid walk in the directed graph associated to~$A$.  A \emph{pointed cycle} is defined to be a self-avoiding closed walk of length at least~$1$, with the ``point'' being the initial/terminal vertex.  We use the term \emph{cycle} to refer to a pointed cycle in which the point is ``forgotten'' (i.e., the cycle is treated as a set of arcs, with no distinguished starting point).
\end{definition}
\begin{definition}
    Given a sequence of matrices $A_1, \dots, A_\atms$ with color set $[\atms]$, a \emph{colored cycle} (or \emph{colored walk}, etc.)\ is a pair $\coloredcyc = (\gamma, j)$ where $\gamma$ is a cycle on~$[n]$ and $j \in [\atms]$ is a color.  The key aspect of this definition is that the \emph{weight} of this colored cycle is defined to be~$w_{A_{j}}(\coloredcyc)$.  We will write this simply as $w(\coloredcyc)$ when $A_1, \dots, A_\atms$ are understood.  We write $\CCyc(A_1, \dots, A_\atms)$ for the collection of all colored cycles of nonzero weight.  (It won't actually matter whether or not we include colored cycles of zero weight, but it is conceptually simpler to exclude them in the unweighted case.)
\end{definition}
\begin{definition}  \label{def:triv-heap-col-cyc}
    Given $A_1, \dots, A_\atms$ as before, a \emph{trivial heap of colored cycles} (cf.~\Cref{def:heap}) is a subset $M \subseteq \CCyc(A_1, \dots, A_\atms)$ in which all colored cycles in~$M$ are pairwise vertex-disjoint.  We define the \emph{length} and \emph{weight} of~$M$ (respectively) to be
    \[
        \length(M) = \sum_{\coloredcyc \in M} \length(\coloredcyc), \qquad w(M) = \prod_{\coloredcyc \in M} w(\coloredcyc).
    \]
    We remark that $\length(M)$ is also the number of vertices that $M$ touches.  We reserve the notation $|M|$ for the number of colored cycles in~$M$.  We write $\THeap(A_1, \dots, A_\atms)$ for the collection of all trivial heaps of colored cycles.
\end{definition}
\begin{example}
    In the monochromatic and unweighted case, a trivial heap of (colored) cycles is a collection of vertex-disjoint directed cycles within a directed graph. Such subgraphs go under many names, such as ``partial $2$-factor'', ``linear subgraph'', or ``sesquilinear subgraph''.
\end{example}
\begin{example}
    In the case where $G$ is an undirected graph with edge \atoms $A_1, \dots, A_\atms$, a trivial heap of colored cycles is just a (partial) matching in~$G$.
\end{example}

\begin{definition}
    Let $A_1, \dots, A_\atms$ be matrices indexed by~$[n]$.  We define their \emph{\factorpolynomial} in indeterminate~$x$ to be
    \begin{equation}    \label{eqn:factor-coeffs}
        \factorpoly{A_1, \dots, A_\atms}{x}  = \sum_{k=0}^n b_k x^{n-k}, \qquad b_k = \sum_{\substack{M \in \THeap(A_1, \dots, A_\atms) \\ \length(M) = k}} (-1)^{|M|} w(M).
    \end{equation}
\end{definition}

\begin{example}                                        \label{eg:charpoly}
    In the monochromatic case, the \factorpolynomial $\factorpoly{A}{x}$ is the same as the characteristic polynomial $\charpoly{A}{x}$.  This is equivalent to what is sometimes called the \emph{Coefficients Theorem for Weighted Digraphs}. It follows easily by expanding the determinant in terms of permutations. %, and is essentially the same as the fact that the coefficient on $x^k$ in $\charpoly{A}{x}$ is $(-1)^k \tr \calC_k(A)$.\rnote{probably $\calC_k$ is not defined yet}
    See~\cite[p.~36]{CDS80} for some history of this fact.
\end{example}
\begin{example}                                        \label{eg:matchingpoly}
    In the case where $G$ is an undirected graph with edge \atoms $A_1, \dots, A_\atms$, the \factorpolynomial $\factorpoly{A_1, \dots, A_\atms}{x}$ is the same as the matching polynomial defined below
    \[
        \matchpoly{G}{x} := \sum_{M \subseteq E, M\in\Matchings(G)} (-1)^{|M|} x^{n-2|M|}
    \]
\end{example}

\begin{definition}
    Let $A$ be an $n \times n$ matrix and let $Q$ be a diagonal $n \times n$ matrix with all diagonal entries from~$\{\pm 1\}$.  Then we call the matrix $Q^{\dagger}\!A Q$ a \emph{balanced edge-signing} of~$A$.  A \emph{random} balanced edge-signing of~$A$ refers to the case of $\bQ^{\dagger\!}A \bQ$, where the diagonal entries of~$\bQ$ are chosen independently and uniformly at random from~$\{\pm 1\}$.
\end{definition}
\begin{remark}  \label{rem:Q-dagger}
    The reader may find it unnecessarily complicated for us to have written $Q^\dagger$ here, since $Q^\dagger = Q$ when $Q$ is diagonal with $\pm 1$ entries.  Also, the suggestion of complex conjugation may look strange given that $A$'s entries are only assumed to be from a commutative ring. However, $A$'s entries will usually be complex numbers or polynomials, and in the future we may sometimes consider conjugating such $A$ by a diagonal matrix~$Q$ whose entries are general complex numbers.  In these cases we will indeed want to write $Q^{\dagger\!}A Q$.
\end{remark}
\begin{remark}
    Consider the unweighted case, when $A$ is the adjacency matrix of a directed graph~$G$. Then a ``balanced edge-signing'' $Q^{\dagger\!}A Q$ is indeed the adjacency matrix of a particular kind of edge-signing of~$G$ called ``balanced'' in the literature~\cite{Zas82}; namely, an edge-signing with the property that the product the signs of the arcs around any directed cycle is~$+1$.
\end{remark}
\begin{remark}  \label{rem:balanced-is-general-for-forests}
    Consider the plain, monochromatic case, when $A$ is the adjacency matrix of a simple undirected graph~$G$.  Suppose also that $G$ is a forest (i.e., it is acyclic, except insofar as an undirected edge is considered to be two opposing directed arcs).  Then \emph{every} possible edge-signing of~$G$ corresponds to a balanced edge-signing; in fact, for each usual edge-signing there are $2^k$ corresponding balanced edge-signings, where $k$ is the number of connected components of~$G$.  Thus in this case, a random balanced edge-signing is equivalent to the usual notion of a uniformly random edge-signing.
\end{remark}

\begin{theorem}                                     \label{thm:factor-poly-comes-from-random-edge-signing}
    Let $A_1, \dots, A_\atms \in \C^{n \times n}$.  Then
    \[
        \factorpoly{A_1, \dots, A_\atms}{x} = \E\bracks*{\charpoly{\bQ_1^{\dagger\!}A_1 \bQ_1 + \cdots + \bQ_\atms^{\dagger\!}A_\atms \bQ_\atms}{x}},
    \]
    where $\bQ_1, \dots, \bQ_\atms$ are independent random balanced edge-signing matrices.
\end{theorem}
\begin{remark}
    As will be seen from the proof of \Cref{thm:factor-poly-comes-from-random-edge-signing}, the expectation is unchanged so long as the random diagonal matrices~$\bQ_j$ have entries $\bQ_j[i,i]$ that are independent complex random variables with mean~$0$ and variance~$1$; for example, each could be chosen uniformly at random from the complex unit circle.
\end{remark}
\begin{remark}
    In the monochromatic case of $\atms = 1$, \Cref{thm:factor-poly-comes-from-random-edge-signing} reduces to the fact that the characteristic polynomial of a matrix~$A$ is invariant to the unitary conjugation $A \mapsto Q^{\dagger\!}A Q$.  In the case that $A_1, \dots, A_\atms$ are the edge \atoms for an undirected graph~$G$, \Cref{thm:factor-poly-comes-from-random-edge-signing} reduces to the Godsil--Gutman theorem~\cite[Corollary~2.2]{GG81} that the matching polynomial of~$G$ is the expected characteristic polynomial of a random edge-signing of~$G$ (here we are using \Cref{rem:balanced-is-general-for-forests}).
\end{remark}

\begin{remark} \label{rem:mohan}
    In the special case when $A_1 = A_2 = \cdots = A_r$, the additive characteristic polynomial $\factorpoly{A, \dots, A}{x}$ (with $r$ copies of~$A$) becomes equivalent to the ``$r$-characteristic polynomial'' introduced by Ravichandran~\cite{Rav16} and notated $\chi_r[A]$ therein.\footnote{This equivalence is particularly clear in the second version of the paper~\cite{LR18}, joint with Leake, which was written around the same time as this paper.  See the discussion just preceding Definition~6.2 in~\cite{LR18}.  We thank Mohan Ravichandran for drawing this to our attention.}  That work --- motivated by Anderson's paving formulation of the Kadison--Singer conjecture --- gave several combinatorial/algebraic formulas for $\chi_r[A]$, showed it is real-rooted for any Hermitian~$A$ using the Interlacing Polynomials method, and gave certain bounds on its roots. 
\end{remark}

\begin{proof}[Proof of \Cref{thm:factor-poly-comes-from-random-edge-signing}]
    Write $\bA =  \bQ_1^{\dagger\!} A_1 \bQ_1 + \cdots + \bQ_\atms^{\dagger\!} A_\atms \bQ_\atms$.  We use the ``Coefficients Theorem'' formula for $\charpoly{\bA}{x}$ mentioned in \Cref{eg:charpoly}:
    \begin{equation}    \label{eqn:thats-charpoly}
        \charpoly{\bA}{x} = \sum_{k=0}^n a_k x^{n-k}, \qquad a_k = \sum_{\substack{M \in \THeap(\bA) \\ \length(M) = k}} (-1)^{|M|} w_{\bA}(M).
    \end{equation}
    Here $M$ runs over all trivial heaps of \emph{uncolored} cycles on~$[n]$, and $w_{\bA}(M)$ refers to the (random) weight of such a trivial heap with respect to~$\bA$.  Let us say that a \emph{cycle-coloring} of~$M$ is any $M' \in \THeap(A_1, \dots, A_\atms)$ obtained by choosing a color in~$[c]$ for each cycle in~$M$.  Comparing \Cref{eqn:thats-charpoly} with the definition of $\factorpoly{A_1, \dots, A_\atms}{x}$, we see it suffices to show for each uncolored~$M$ that
    \begin{equation}    \label{eqn:cycle-col}
        \E\bracks*{w_{\bA}(M)} = \sum_{\text{cycle-colorings } M' \text{ of } M} w(M'),
    \end{equation}
    where $w(M')$ above is with respect to $A_1, \dots, A_\atms$, as in \Cref{def:triv-heap-col-cyc}.  Now
    \[
        w_{\bA}(M) = \prod_{\coloredcyc \in M} w_{\bA}(\coloredcyc) = \prod_{\coloredcyc \in M} \prod_{\substack{\text{arcs} \\ e \in \coloredcyc}} w_{\bA}(e) = \prod_{\coloredcyc \in M} \prod_{\substack{\text{arcs} \\ e \in \coloredcyc}} \parens*{A_1[e] \bq_1[e] + \cdots + A_\atms[e] \bq_\atms[e]},
    \]
    where for $e = (i,i')$ we have used the shorthands $A[e] = A[i,i']$ and $\bq_j[e] = \bQ_j[i,i]^* \bQ_j[i',i']$.  Expanding out the above product yields
    \[
        w_{\bA}(M) = \sum_{\substack{\text{arc-colorings} \\ \chi : \{\text{arcs in }M\} \to [\atms]}} \ \prod_{\coloredcyc \in M} \prod_{\substack{\text{arcs} \\ e \in \coloredcyc}} A_{\chi(e)}[e]\bq_{\chi(e)}[e].
    \]
    Consider the expectation of a particular term in the above sum, corresponding to some arc-coloring~$\chi$.
    Using the fact that the random variables $\bQ_j[i,i]$ are independent with mean~$0$ and variance~$1$, the expectation is~$0$ unless the $\bQ_j[i,i]$'s that appear appear in pairs (as $\bQ_j[i,i]^* \bQ_j[i,i]$), in which case it equals $\prod\{A_{\chi(e)}[e] : e \text{ in } M\}$.  This sort of pairing-up occurs if and only if for each cycle $\coloredcyc \in M$, the arc-coloring $\chi$ assigns the same color to each arc in~$\coloredcyc$; i.e., if and only if $\chi$ agrees with some cycle-coloring~$M'$ of~$M$.  In this case, the contribution $\prod\{A_{\chi(e)}[e] : e \text{ in } M\}$ indeed equals $w(M')$.  Thus we have established \Cref{eqn:cycle-col}, completing the proof.
\end{proof}

Once we have \Cref{thm:factor-poly-comes-from-random-edge-signing} in hand, the following fact is easy to prove.
\begin{fact}    \label{fact:factor-factorpoly}
    Suppose $H = A_1 + \cdots + A_\atms$ is a sum graph and there is a way to partition $[\atms]$ into $S_1,\dots, S_t$ such that $\sum_{j\in S_i}A_j$ is a separate connected component for each $j$. Then
    \[
        \factorpoly{A_1,\dots, A_\atms}{x} = \prod_{i = 1}^t \factorpoly{(A_j)_{j\in S_i}}{x}
    \]
\end{fact}

\subsection{Freelike walks}     \label{sec:freelike}
\begin{definition}
    Let $\omega = (u_0, \dots, u_T)$ be a walk on vertex set~$[n]$.  There is a natural way (``loop-erasing'') of decomposing~$\omega$ into a self-avoiding walk~$\eta$, together with a collection of cycles.  We give an abbreviated description of it here; see also Godsil's description of it~\cite[Sec.~6.2]{God93}.  We follow the walk $u_0, u_1, u_2, \dots$ until the first repetition of a vertex; say $u_s = u_t$ with $s < t$.  We call the cycle $(u_s, u_{s+1}, \dots, u_t = u_s)$ thus formed the first \emph{piece} in the walk.  We then delete this piece from~$\omega$ (leaving one occurrence of~$u_s$), and repeat the process, starting again from~$u_0$.  This generates a sequence of pieces (cycles).  We proceed until the walk has no more repeated vertices, at which point the remaining self-avoiding walk (possibly of length~$0$) is termed~$\eta$.
\end{definition}
\begin{remark}
    We will be most interested in closed walks~$\omega$, in which case the self-avoiding walk~$\eta$ indeed degenerates to the initial/terminal point $u_0$ of~$\omega$.
    %In the decomposition of such a closed walk, it is helpful to think of the pieces as being arranged into a ``tree of cycles'', rooted at~$u_0$.\rnote{Put in a figure!!XXX}  Each node in this tree is piece from the decomposition, thought of as a cycle on~$[n]$ (with the root $u_0$ considered a degenerate cycle of length~$0$).  The children of a given piece~$\gamma$ are an \emph{ordered} sequence of cycles, $\gamma_1, \dots, \gamma_s$, with the ``point'' of each $\gamma_i$ being a vertex that $\gamma_i$ has in common with~$\gamma$. One thinks of the children as being attached to their parent at their points, and the ordering of the children should be consistent with the ordering of their points along~$\gamma$, when $\gamma$ is traversed starting from \emph{its} point.  The overall tree thus formed looks like a ``directed cactus graph'', and the closed walk~$\omega$ is recovered by taking a kind of depth-first search in the tree, starting from~$u_0$: For a given node~$\gamma$, one traverses it starting from its point, recursively traversing each child subtree in order when once its attachment point is reached.
\end{remark}
\begin{definition}
    Given matrices $A_1, \dots, A_\atms$ indexed by~$[n]$, we define a \emph{closed freelike\footnote{Please excuse the pun mixing Voiculescu's ``free'' with Godsil's ``treelike walks''.} walk} to be a closed walk~$\freelikewalk$ on~$[n]$ in which each piece is assigned a color from~$[\atms]$.  Its \emph{weight} $w(\freelikewalk)$ is defined to be the product of~$w(\coloredcyc)$ over all pieces (colored cycles) $\coloredcyc$ in~$\freelikewalk$.  (Recall that if $\coloredcyc = (\gamma, j)$ is a colored cycle, \emph{its} weight is $w_{A_j}(\gamma)$.)  We write $\FrWalk(A_1, \dots, A_\atms)$ for the collection of all closed freelike walks of nonzero weight.
\end{definition}
\begin{remark}
    In the unweighted case, when each $A_j$ is the adjacency matrix of a simple directed graph $G_j$ on common vertex set~$[n]$, an element $\freelikewalk \in \FrWalk(A_1, \dots, A_\atms)$ is a closed walk on~$[n]$ in which each piece has been assigned to an \atom $G_j$ in which it wholly appears.
\end{remark}
\begin{remark}
    In the case when $A_1, \dots, A_\atms$ are the edge \atoms of an undirected graph $G$, an element $\freelikewalk \in \FrWalk(A_1, \dots, A_\atms)$ is equivalent to a closed walk within~$G$ in which each piece has length~$2$; this is precisely the definition from Godsil~\cite{God81} of a closed \emph{treelike} walk in~$G$.
\end{remark}

Consider the monochromatic, unweighted case, when $A \in \C^{n \times n}$ is the adjacency matrix of a simple directed graph~$G$.  A common way to study the roots $\lambda_1, \dots, \lambda_n$ of the characteristic polynomial $\charpoly{A}{x}$ (i.e., the eigenvalues of~$A$) is via the \emph{Trace Method}, which says that the $k$th~\emph{power sum} of these roots, $p_k(\lambda) = \sum_{i=1}^n \lambda_i^k$, is equal to the number of closed walks in~$G$ of length~$k$.  This may be seen as a combinatorial interpretation of the characteristic polynomial of a graph.

Similarly, Godsil gave a combinatorial interpretation of the matching polynomial of an undirected graph~$G$: he showed~\cite[Theorem~3.6(b)]{God81} that the $k$th~power sum of the roots of $\matchpoly{G}{x}$ is equal to the number of closed \emph{treelike} walks in~$G$.

In \Cref{thm:counts-treelike} below, we give a common generalization of these two facts: in the unweighted case, it says that for a sum graph $G = A_1 + \cdots + A_\atms$, the $k$th~power sum of the roots of $\factorpoly{A_1, \dots, A_\atms}{x}$ is equal to the number of closed \emph{freelike} walks in~$G$.  Indeed, we prove this in the general weighted case, where the $n \times n$ matrices~$A_j$ have entries from a commutative ring.  In this case it doesn't make sense to speak of eigenvalues, but we may still recall a sensible interpretation of the ``power sum $p_k(\roots(f))$'' via the theory of generating functions.  If
\[
    f(x) = (x-\lambda_1)(x-\lambda_2) \cdots (x-\lambda_n)
\]
is a general degree-$n$ monic complex polynomial with roots $\lambda_1, \dots, \lambda_n$, then
\begin{align}
    \frac{f'(x)}{f(x)} = \frac{d}{dx} \log f(x) = \sum_{i=1}^n \frac{1}{x-\lambda_i} = x^{-1} \sum_{i=1}^n \frac{1}{1- \lambda_ix^{-1}} &= x^{-1} \sum_{i=1}^n \parens*{1 + \lambda_ix^{-1}  + \lambda_i^2x^{-2}   + \cdots} \nonumber \\
    &= x^{-1}\parens*{n + p_1(\lambda)x^{-1}   +  p_2(\lambda) x^{-2}  + \cdots}, \label{eqn:powersumgen}
\end{align}
at least at the level of formal generating functions.  Thus for any monic polynomial $f(x)$ with coefficients in a commutative ring, the desired interpretation of the $k$th power sum of its roots is
\begin{equation}    \label{eqn:pk}
    p_k = p_k(\roots(f)) = [x^k]\parens*{x^{-1} \frac{f'(x^{-1})}{f(x^{-1})}}.
\end{equation}
Furthermore, let us define $e_k = e_k(\roots(f))$, the ``$k$th elementary symmetric polynomial $f$'s roots'', in the natural way; namely, via $f$'s coefficients, as
\begin{equation}    \label{eqn:ek}
    f(x) = x^n + \sum_{k=1}^n (-1)^{k} e_{k} x^{n-k}.
\end{equation}
Then from \Cref{eqn:pk} and \Cref{eqn:ek} one easily infers \emph{Newton's identities}, which give an alternative, recursive definition for~$p_k(\roots(f))$ in terms of $f$'s coefficients:
\begin{equation}    \label{eqn:newton}
        p_k = (-1)^{k+1} k e_k + \sum_{i=1}^{k-1} (-1)^{i+1} e_{i}p_{k-i}.
\end{equation}

\begin{theorem}                                     \label{thm:counts-treelike}
    Let $A_1, \dots, A_\atms$ be matrices indexed by~$[n]$.  Then writing $\factorpolyplain = \factorpoly{A_1, \dots, A_\atms}{x}$, it holds for all $k \in \N^+$ that
    \begin{equation}    \label{eqn:its-p}
        p_k(\roots(\factorpolyplain)) = \sum_{\substack{\freelikewalk \in \FrWalk(A_1, \dots, A_\atms) \\ \length(\freelikewalk) = k}} w(\freelikewalk).
    \end{equation}
\end{theorem}
We will give two proofs of \Cref{thm:counts-treelike}.\footnote{in fact, probably the two proofs are more or less the same}  On one hand, one might say that \Cref{thm:counts-treelike} follows almost immediately from the ``logarithmic lemma'' in Viennot's theory of ``heaps of pieces'' (partially published in \cite{Vie86} and described in more detail in the YouTube series \cite{Vie17youtube}).  On the other hand, this theory is not completely published, and we find it worthwhile to give the following direct proof for the sake of being self-contained:
\begin{proof}[Proof of \Cref{thm:counts-treelike}]
    Let us write
    \[
        \wt{p}_k = \sum_{\substack{\freelikewalk \in \FrWalk(A_1, \dots, A_\atms) \\ \length(\freelikewalk) = k}} w(\freelikewalk).
    \]
    We will show that $\wt{p}_k$ satisfies the recursion in \Cref{eqn:newton} vis-a-vis the coefficients of~$\factorpolyplain$ (cf.~\Cref{eqn:factor-coeffs,eqn:ek}); i.e., vis-a-vis
    \[
        e_k = e_k(\roots(\factorpolyplain)) = (-1)^k \sum_{\substack{M \in \THeap(A_1, \dots, A_\atms) \\ \length(M) = k}} (-1)^{|M|} w(M).
    \]
    It will then follow that $p_k(\roots(\factorpolyplain)) = \wt{p}_k$. %The fact that $\wt{p}_1 = e_1$ is clear; both are equal to $\sum_{j=1}^\atms \tr A_j$.

    Define $\Psi_k$ to be the set of pairs $(\freelikewalk, M) \in \FrWalk(A_1, \dots, A_\atms) \times \THeap(A_1, \dots, A_\atms)$ satisfying $\length(\freelikewalk) + \length(M) = k$.  \emph{Exception:} in the case of $\length(\freelikewalk) = 0$ and $\length(M) = k$, we only include those pairs for which $\freelikewalk$'s single vertex appears in~$M$.

    We define an involution $\psi$ on $\Psi_k$ as follows.  Given $(\freelikewalk, M) \in \Psi_k$, let $\coloredcyc_1$ be the first piece in~$\freelikewalk$ (or if $\length(\freelikewalk) = 0$, define $\coloredcyc_1 = \freelikewalk$).  There are now two cases.
    \paragraph{Case 1:} \emph{The initial part of~$\freelikewalk$, up to and including the traversal of~$\coloredcyc_1$, visits a vertex appearing in~$M$.}  (In the exceptional case of~$\length(\freelikewalk) = 0$, this case always occurs.)  In this case, let $v$ be the earliest such vertex and let $\coloredcyc$ be the (unique) piece in~$M$ containing it.  Since~$v$ is earliest, no other vertex of~$\coloredcyc$ occurs in~$\freelikewalk$ prior to this~$v$; hence we may form a new freelike walk $\freelikewalk^+$ by inserting~$\coloredcyc$ into~$\freelikewalk$ just after the first occurrence of~$v$.  We define $\psi(\freelikewalk,M) = (\omega^+, M \setminus \coloredcyc)$.

    \paragraph{Case 2:} \emph{The initial part of~$\freelikewalk$, up to and including the traversal of~$\coloredcyc_1$, is vertex-disjoint from~$M$.}  In this case, we let $\freelikewalk^-$ be the freelike walk formed from~$\omega$ by deleting its first piece~$\coloredcyc_1$, and we define $\psi(\freelikewalk, M) = (\freelikewalk^-, M \sqcup \coloredcyc_1)$.  The union~$M \sqcup \coloredcyc_1$ is indeed vertex-disjoint, since we are in Case~2.  (Also, in the exceptional case that $\length(\freelikewalk^-) = 0$, we indeed have that $\freelikewalk^-$'s single vertex appears in $M \sqcup \coloredcyc_1$, since it's in~$\coloredcyc_1$.)\\

    We now verify that~$\psi$ is indeed an involution.  Suppose first that $(\freelikewalk,M) \in \Psi_k$ falls into Case~1.  Using the terminology from that case,  note that~$v$ occurs in~$\freelikewalk$ earlier than the completion of~$\coloredcyc_1$; thus $\coloredcyc$ occurs in~$\freelikewalk^+$ earlier than~$\coloredcyc_1$ and hence it is the first piece in walk~$\freelikewalk^+$.  Now it is not hard to see that $\psi(\freelikewalk,M) = (\freelikewalk^+,M \setminus \coloredcyc)$ will fall into Case~2 (with the role of $\coloredcyc_1$ being played by~$\coloredcyc$), and $\psi(\freelikewalk^+,M \setminus \freelikewalk)$ will again be~$(\freelikewalk,M)$.

    On the other hand, suppose $(\freelikewalk,M) \in \Psi_k$ falls into Case~2.  Let $v$ be the vertex at which $\freelikewalk$ enters~$\coloredcyc_1$ for the first time.  Since we are in Case~2, no piece in~$M$ touches a vertex occurring earlier than~$v$ in~$\freelikewalk$ (and nor does any vertex in~$\coloredcyc_1$ other than~$v$ occur earlier).  Thus in considering $\psi(\freelikewalk,M) = (\freelikewalk^-, M\sqcup \coloredcyc_1)$, we see that this pair will fall into Case~1: vertex~$v$ (which is in~$\coloredcyc_1$) will appear in~$\freelikewalk^-$ prior to the completion of its first piece, and $\psi(\freelikewalk^-, M \sqcup \coloredcyc_1)$ will indeed be formed by reinserting~$\coloredcyc_1$ into~$\freelikewalk^-$ at the first occurrence of~$v$.

    With~$\psi$ in hand, let us define the \emph{weight} of pair $(\freelikewalk,M)$ to be $W(\freelikewalk,M) = w(\freelikewalk)(-1)^{|M|}w(M)$.  It is easy to see from its definition that $W(\psi(\freelikewalk,M)) = -W(\freelikewalk,M)$.  Thus since $\psi$ is an involution,
    \[
        \sum_{(\freelikewalk,M) \in \Psi_k} W(\freelikewalk,M) = 0,
    \]
    since the summands cancel in pairs.  Expanding the left-hand side in terms of its definitions, we get
    \begin{align*}
        0 &=  k \cdot \sum_{\substack{M \in \THeap(A_1, \dots, A_\atms) \\
        \length(M)  = k}}  (-1)^{|M|} w(M)  + \sum_{\substack{\freelikewalk \in \FrWalk(A_1, \dots, A_\atms) \,\ M \in \THeap(A_1, \dots, A_\atms) \\
        \length(\freelikewalk) + \length(M)  = k, \ \length(\freelikewalk)\neq 0}} w(\freelikewalk)   (-1)^{|M|}  w(M)    \\
        &= (-1)^k k e_k + \parens*{\wt{p}_k + \sum_{i=1}^{k-1} (-1)^{i} e_i \wt{p}_{k-i}},
    \end{align*}
    which indeed shows that $\wt{p}_k$ satisfies \Cref{eqn:newton} (Newton's identities), completing the proof.
\end{proof}

\subsubsection{Heaps of pieces}
We now show how \Cref{thm:counts-treelike} also follows from Viennot's theory of ``heaps of pieces'' (see also \cite{Kra06}, \cite{CF06} and \cite{GR17}).  We instantiate the theory as follows:

%log sum heaps w(E) t^len(E) = sum_P pointed pyramid t^ell(P)/ell(P) = sum_free likes t^(ell())

\begin{definition}  \label{def:heap}
    Given $A_1, \dots, A_\atms$ as before, a \emph{heap of colored cycles}~$H$ is a finite collection of ``pieces'' satisfying certain conditions.  Each ``piece'' is a pair $(\coloredcyc, h)$, where $\coloredcyc \in \CCyc(A_1, \dots, A_\atms)$ is a colored cycle and  $h \in \N$ is the ``height'' or ``level''.  Two colored cycles $\coloredcyc_1, \coloredcyc_2$ are said to be ``dependent'' (written $\coloredcyc_1 \calC \coloredcyc_2$) if they have a vertex in common.  The heap conditions are the following:
    \begin{itemize}
        \item If two pieces are at the same height, they are independent. More precisely, for $(\coloredcyc_1, h_1), (\coloredcyc_2, h_2) \in H$, if $\coloredcyc_1 \calC \coloredcyc_2$ then $h_1 \neq h_2$.
        \item If a piece is not at ground level, then it is supported by another piece.  More precisely, if $(\coloredcyc, h) \in H$ with $h > 0$ then there exists $(\coloredcyc', h-1) \in H$ with $\coloredcyc \calC \coloredcyc'$.
    \end{itemize}
    We write $\Heap(A_1, \dots, A_\atms)$ for the collection of all heaps of colored cycles. (Note that trivial heaps are ones in which all heights are~$0$.)  Finally, if $H \in \Heap(A_1, \dots, A_\atms)$, we define its ``valuation'' to be
    \[
        v(H) = w(H)x^{\length(H)}, \qquad \text{where } \length(H) = \sum_{(\coloredcyc,h) \in H} \length(\coloredcyc);
    \]
    here $x$ is an indeterminate.
\end{definition}
In the heaps of pieces setup, Viennot (see \cite{Vie86}) showed the following generating function identity:
\paragraph{Inversion Lemma.}  $\displaystyle
    \sum_{H \in \Heap(A_1, \dots, A_\atms)} v(H) = \frac{1}{\displaystyle \sum_{M \in \THeap(A_1, \dots, A_\atms)} (-1)^{|M|} v(M)}.$\\

From \Cref{eqn:factor-coeffs}, we can easily recognize the right-hand side above as $\parens*{x^n \factorpolyplain(1/x)}^{-1}$.  If we now apply the operator $x \frac{d}{dx} \log$ to both sides, \Cref{eqn:powersumgen} lets us quickly deduce the generating function
\begin{equation}    \label{eqn:powersumgen2}
    x \frac{d}{dx} \log \parens*{ \sum_{H \in \Heap(A_1, \dots, A_\atms)} v(H) } = p_1 x + p_2 x^2 + p_3 x^3 + \cdots.
\end{equation}
Now we may apply the other main identity in Viennot's theory \cite[Chapter 2d]{Vie17youtube}:
\paragraph{Logarithmic Lemma.}  $\displaystyle x \frac{d}{dx} \log \parens*{ \sum_{H \in \Heap(A_1, \dots, A_\atms)} v(H) } = \sum_{P \in \Pyr(A_1, \dots, A_\atms)} v(P).$\\

Here $\Pyr(A_1, \dots, A_\atms)$ denotes all \emph{pyramids of colored cycles}; that is, heaps of colored cycles having a unique \emph{maximal} piece (meaning a unique piece supporting no other pieces) along with a distinguished vertex $v$ contained in the maximal piece.  Finally, an important aspect of the theory, that can be found in \cite[Chapter 3a, 3b]{Vie17youtube} is that there is bijection between pyramids of cycles and closed walks (of positive length). In our case, this is a bijection between pyramids  of colored cycles, $P$, and closed walks in which each piece is colored --- in other words, freelike walks $\freelikewalk$ (of positive length).\footnote{Very briefly:  Given a pyramid with a distinguished starting vertex in its maximal piece, one obtains the freelike walk by always following the ``lowest'' arc in the pyramid, emanating from the current vertex, that has not yet been followed. Conversely, given the freelike walk with distinguished starting vertex, one forms the pyramid by ``dropping in'' the colored pieces as they are encountered in the loop-erasing decomposition.}  And further, the bijection preserves the set of colored cycles used; hence, $v(P) = w(\freelikewalk) x^{\length(\freelikewalk)}$ under this bijection.  Combining this bijection with the Logarithmic Lemma and \Cref{eqn:powersumgen2} we conclude
\[
    \sum_{\substack{\freelikewalk \in \FrWalk(A_1, \dots, A_\atms) \\ \length(\freelikewalk) \neq 0} } w(\freelikewalk) x^{\length(\freelikewalk)} = p_1 x + p_2 x^2 + p_3 x^3 + \cdots,
\]
which is the generating function form of \Cref{thm:counts-treelike}.

\subsection{Root bounds for the \factorpolynomial}

Given the expected-characteristic-polynomial formula of \Cref{thm:factor-poly-comes-from-random-edge-signing}, we may immediately apply a theorem of Hall, Puder, and Sawin~\cite[Thm.~4.2]{HPS18} (see also~\cite[Thm.~3.3]{MSS15d}) to deduce the below~\Cref{thm:its-real-rooted}.  This theorem generalizes the fact that the characteristic polynomial of a Hermitian matrix is real-rooted, and the theorem~\cite{HL72} that the matching polynomial of an undirected graph is real-rooted.
\begin{theorem}                                     \label{thm:its-real-rooted}
    Let $A_1, \dots, A_\atms \in \C^{n \times n}$ be Hermitian.  Then the \factorpolynomial $\factorpoly{A_1, \dots, A_\atms}{x}$ is real-rooted.
\end{theorem}
This real-rootedness property can be seen as a corollary of \Cref{thm:lifts-interlacing-family}, which is discussed in \Cref{sec:additive-lift-interlacing}.

We would now like to generalize the theorem~\cite{HL72} that for a $\Delta$-regular graph, the roots of the matching polynomial have magnitude at most~$2\sqrt{\Delta-1}$; and more generally, the theorem~\cite{God81,MSS15a} that the roots of $\matchpoly{G}{x}$ have magnitude at most $\specrad(\UCT(G))$.
\begin{theorem}                                     \label{thm:root-bound}
    Let $H = A_1 + \cdots + A_\atms$ be a connected sum graph and  $X = A_1 \addprod \cdots \addprod A_\atms$ be the \emph{\additiveproduct} of atoms $A_1, \cdots, A_\atms$ on vertex set~$[n]$, as in \Cref{def:additive-product}.  Then the roots of $\factorpoly{A_1, \dots, A_\atms}{x}$ lie in the interval $[-\specrad(X), \specrad(X)]$.
\end{theorem}
\begin{proof}
    Let $\lambda > 0$ denote the largest magnitude among the $n$ roots of $\factorpoly{A_1, \dots, A_\atms}{x}$.  From \Cref{thm:counts-treelike} it follows that the number of closed freelike walks $\upomega \in \FrWalk(A_1, \dots, A_\atms)$ of length~$2k$ is at least~$\lambda^{2k}$.  Hence there exists a vertex $i \in [n]$ such that the number of closed length-$2k$ freelike walks that start and end at~$i$ is at least $\frac{1}{n} \lambda^{2k}$. Let $u$ be a vertex that maps to $i$ in a covering map from vertices from $X$ to $[n]$.  Thus, the number of closed walks starting and ending at $u$ in $X$ of length $2k$ is at least $\frac{1}{n}\lambda^{2k}$.

    In other words (using the notation from \Cref{def:specrad}), $c^{(2k)}_{uu} \geq \frac{1}{n} \lambda^{2k}$ for every $k \in \N$.  It follows immediately from \Cref{fact:specrad-basic} that $\specrad(X) \geq (\frac{1}{n})^{1/2k} \lambda$ for every~$k$, and hence $\lambda \leq \specrad(X)$ as desired.
\end{proof}

\begin{remark}  \label{rem:root-bound-factor}
    Suppose $H = A_1 + \cdots + A_\atms$ is a disconnected graph where each of the $t$ connected components partition $[c]$ into $S_1,\dots, S_t$ such that the $i$th connected component is equal to $\sum_{j\in S_i} A_j$. Let $X_i$ denote the additive product of atoms in $S_i$. Combining \Cref{fact:factor-factorpoly} and \Cref{thm:root-bound} gives us that $\factorpoly{A_1,\dots,A_\atms}{x}$ has all its roots in $[-\max_{1\le i \le t}\specrad(X_i), \max_{1\le i \le t}\specrad(X_i)]$.
\end{remark}

\begin{remark}
    Godsil showed in \cite{God81} that the matching polynomial of a graph $G$ divided the characteristic polynomial of a certain subgraph of the universal cover of $G$, known as the \emph{path tree} of $G$, from which the desired root bounds on the matching polynomial of $G$ follow.  However, an analogous divisibility result for the additive characteristic polynomial $\factorpoly{A_1,\dots,A_\atms}{x}$ of sum graph $H = A_1 + \cdots + A_c$ seems elusive, which motivates studying the moments of $\factorpoly{A_1,\dots,A_\atms}{x}$ via other combinatorial means.
\end{remark}

\section{Interlacing Families}      \label{sec:interlacing-families}

In this section, we give background and facts about interlacing families which are proved in \cite{MSS15a,MSS15b}.

\begin{definition}
Let $p$ and $q$ be real rooted polynomials with $\deg(p)=\deg(q)+1$. Denote the $i$-th largest root of $p$ and $q$ with $\lambda_i$ and $\mu_i$ respectively. We say $q$ \emph{interlaces} $p$ if $\lambda_i\ge\mu_i\ge\lambda_{i+1}$ for $1\leq i\leq\deg(q)$.
\end{definition}

\begin{definition}
We say that a family of polynomials $\calF$ has a \emph{common interlacing} if there is some polynomial $r$ that interlaces every polynomial in $\calF$.
\end{definition}

\begin{theorem}     \label{thm:convex-interlace-real-rooted}
A family of polynomials $\calF$ has a common interlacing if and only if $\E_{\bp\sim\calD}[\bp]$ is real-rooted for any distribution $\calD$ over $\calF$.
\end{theorem}

\begin{definition}
In our context, we call a distribution $\calD$ over binary strings of length $M$ a \emph{product distribution} if it is distributed as $\bb_1\bb_2\dots\bb_M$ for independent $\bb_i$ where each $\bb_i\sim\mathrm{Bernoulli}(p_i)$.
\end{definition}

\begin{definition}
Let $\calF$ be a family of real rooted polynomials indexed by $\{0,1\}^m$. We call $\calF$ an \emph{interlacing family} if the polynomial given by $\E_{(\bs_1,\bs_2,\ldots,\bs_m)\sim\calD} \left[p_{s_1,s_2,\ldots,s_m}\right]$ is real rooted for all product distributions $\calD$ over $\{0,1\}^m$.
\end{definition}

\begin{remark}
    For the sake of intuition, we find it fruitful to give an equivalent definition of an interlacing family of polynomials that was given in \cite{MSS15a}. Consider a binary tree of depth $m$ where a vertex $v$ is labeled by a binary string representing the path from the root to $v$. The leaves, thus, are labeled with $m$-bit binary strings. Now, suppose we place a polynomial on each leaf, and recursively fill in vertices in the tree with polynomials by choosing $p_v$ we place at $v$ as an arbitrary convex combination of $p_{u}$ and $p_{u'}$ at the children of $v$. Then, the family $\{p_\ell\}_{\ell\in\mathsf{Leaves}}$ is an interlacing family if every pair of polynomials that share a common parent in the constructed tree has a common interlacing.
\end{remark}

\begin{theorem}     \label{thm:root-bound-interlace}
Suppose $\calF$ is an interlacing family, then for any product distribution $\mathcal{D}$ over $\{0,1\}^m$, there is $s_1^*,\ldots,s_m^*$ such that
\[\maxroot\left(\E_{(s_1,\ldots,s_m)\sim\mathcal{D}}\left[p_{s_1,\ldots,s_m}\right]\right)\geq \maxroot(p_{s_1^*,\ldots,s_m^*})\]
\end{theorem}

\section{Random additive lifts and the interlacing property}   \label{sec:additive-lift-interlacing}

\begin{definition}
    Let $V$ be a vertex set. We define
    \[
        \BitToSwap_{ij}(b) :=
        \begin{cases}
            \tau_{ij} &\text{if $b=1$}\\
            1 &\text{if $b=0$}
        \end{cases}
    \]
    where $\tau_{ij}$ is the transposition that swaps $i$ and $j$, and $1$ is the identity permutation.
    For $x$ in $\{0,1\}^{\nlift \choose 2}$ indexed by $ij$ with $1\leq i<j \leq \nlift$, define
    \[
        \StringToPerm(x) :=
            \prod_{i=1}^{n-1}\prod_{j=i+1}^n \BitToSwap_{ij}(x_{ij})
    \]
    And finally for $y$ in $\LiftEnc_{\nlift,V}:=\left(\{0,1\}^{\nlift\choose 2}\right)^V$ let $\StringToPotential(y)$ be the potential $Q$ such that $Q(v) = \StringToPerm(y_v)$.
\end{definition}

\begin{definition}
    Given a sum graph $H = A_1 + \cdots + A_\atms$ on vertex set $V$, $z\in \LiftEnc_{\nlift,V}^\atms$ and $\pi$ a representation of $\symm{\nlift}$, following \Cref{def:additive-lift} we use the notation
    \[
        \LiftAdj{z}{\pi}(H) := \sum_{j=1}^\atms \LiftAdj{\StringToPotential(z_j)}{\pi}(A_j)
    \]
\end{definition}

\begin{definition}
    Following \cite{HPS18}, we call a $\C^{\repdim \times \repdim}$-valued random matrix~$\bT$ a \emph{rank-$1$ random variable} if $\bT$ is distributed as $U^\dagger \diag(\bomega, 1, \dots, 1) V$ for some fixed $U,V \in \Unitaries{\repdim}$ and some random variable~$\bomega$ taking values in the complex unit circle.
\end{definition}
The following facts are also simple:
\begin{fact}                                        \label{fact:stoch-refl}
    Let $\pi \in \{\permrep, \stdrep, \sgnrep\}$ be a representation of $\symm{\nlift}$ and let $\tau \in \symm{\nlift}$ be a transposition.  Then $\pi(\tau)$ is unitarily conjugate to $\diag(-1, 1, 1, \dots, 1)$.  Thus if $\bb$ is a Bernoulli random variable, then $\pi(\btau)$ is a rank-1 random variable.
\end{fact}

\begin{fact}        \label{fact:prod-refl}
    Subsequently, if $\bx$ is drawn from a product distribution over $\{0,1\}^{\nlift\choose 2}$, then $\pi(\StringToPerm(\bx))$ is the product of ${\nlift \choose 2}$ independent rank-1 random variables.
\end{fact}

\begin{fact}                                        \label{fact:perm-from-transpo}
    There is a product distribution $\calD$ over $\{0,1\}^{\nlift\choose 2}$ such that $\StringToPerm(\bz)$ for $\bz$ drawn from $\calD$ is uniform in $\symm{\nlift}$ (see, e.g., \cite[Remark~4.7]{HPS18}).
\end{fact}

Let $\charpoly{M}{x} = \det(x\Id - M)$ denote the characteristic polynomial, in indeterminate~$x$, of matrix~$M$.
\begin{theorem}         \label{thm:lifts-interlacing-family}
    Let $H = A_1 + \cdots + A_\atms$ be a sum graph on vertex set~$V = [\vtcs]$, and let $\pi \in \{\sgnrep, \stdrep, \permrep\}$ be a representation of $\symm{\nlift}$.
    Then $\left\{\charpoly{\LiftAdj{z}{\pi}(H)}{x}\right\}_{z\in \LiftEnc_{\nlift,V}^\atms}$ is an interlacing family.
\end{theorem}
\begin{proof}
    It suffices to prove that $\E_{\bz\sim\calD}\left[\charpoly{\LiftAdj{\bz}{\pi}(H)}{x}\right]$ is real-rooted for any product distribution $\calD$ over $\LiftEnc_{\nlift,V}^\atms$. We have
    \begin{equation} \label{eqn:big-lift}
        \LiftAdj{\bz}{\pi}(H) = \sum_{j=1}^\atms \pi^{\boxtimes V}(\StringToPotential(\bz_j))^\dagger \cdot \Adj_\repdim(A_j) \cdot \pi^{\boxtimes V}(\StringToPotential(\bz_j))
    \end{equation}
    where we can express each $\StringToPotential(\bz_j)$ as
    \[
        (\StringToPerm(\bz_{j,1}), 1, \dots, 1) \cdot (1, \StringToPerm(\bz_{j,2}), \dots, 1) \cdot \cdots \cdot (1, 1, \dots, \StringToPerm(\bz_{j,\vtcs})).
    \]

    Denoting the $i$-th term of the above product with $\brho_{j, i}$, \Cref{eqn:big-lift} can be reexpressed as
    \begin{align*}
        \LiftAdj{\bz}{\pi}(H) &= \sum_{j=1}^\atms
            \pi^{\boxtimes V}(\brho_{j, 1} \cdots \brho_{j, \vtcs})^\dagger \cdot \Adj_\repdim(A_j) \cdot \pi^{\boxtimes V}(\brho_{j,1} \cdots \brho_{j, \vtcs}) \\
            &= \sum_{j=1}^\atms
            \pi^{\boxtimes V}(\brho_{j, \vtcs})^\dagger \cdots \pi^{\boxtimes V}(\brho_{j,1})^\dagger \cdot \Adj_\repdim(A_j) \cdot \pi^{\boxtimes V}(\brho_{j, \vtcs}) \cdots \pi^{\boxtimes V}(\brho_{j, \vtcs}).
    \end{align*}
    Now each $\pi^{\boxtimes V}(\brho_{j,k})$ is block-diagonal, with all blocks identity except for a single block that is $\pi(\StringToPerm(\bz_{j,k}))$.  By \Cref{fact:prod-refl}, the exceptional block is in fact the product of ${\nlift \choose 2}$ independent rank-$1$ random variables; thus $\pi^{\boxtimes V}(\brho_{j,k})$ can also be written as a product $\bT_{j,k,1}\cdots\bT_{j,k,{\nlift\choose 2}}$ of rank-$1$ random variables~$\bT_{j,k,t}$. So
    \[
        \LiftAdj{\bz}{\pi}(H) = \sum_{j=1}^\atms
           \bT_{j,\vtcs,{\nlift\choose 2}}^\dagger \cdots \bT_{j,1,1}^\dagger \cdot \Adj_\repdim(A_j) \cdot \bT_{j,1,1} \cdots \bT_{j,\vtcs,{\nlift\choose 2}},
    \]
    and real-rootedness of $\E_{\bz\sim\calD}\left[\charpoly{\LiftAdj{\bz}{\pi}(H)}{x}\right]$ is now a consequence of \cite[Theorem 4.2]{HPS18}.
\end{proof}

\section{Root Bounds for random additive $\nlift$-lifts}
\label{sec:minors-root-bounds}

Let $\Gamma$ be a finite group, $\pi$ be a representation of $\Gamma$, $A_1,\dots,A_\atms$ be Hermitian matrices in $\C^{nd\times nd}$, and let $\bQ_1,\dots,\bQ_\atms$ be independent and uniformly random $\Gamma$-potentials on $[n]$. Our object of study this section is $\E[\charpoly{\bcalA}{x}]$ where
\begin{align} \label{eq:random-lift-represented}
    \bcalA := \sum_{j=1}^c \pi^{\boxtimes V}(\bQ_j)^\dagger\cdot A_j \cdot \pi^{\boxtimes V}(\bQ_j)
\end{align}

Recall that in \Cref{sec:factor-poly}, we established root bounds on the \factorpolynomial, which by \Cref{thm:factor-poly-comes-from-random-edge-signing} is equal to $\E[\charpoly{\bcalA}{x}]$ when $\Gamma = S_2$ and $\pi = \stdrep$. In this section, we extend those same root bounds to further $(\Gamma,\pi)$ pairs, specifically pairs satisfying ``Property ($\calP 1$)'' that is defined in \Cref{def:Property-P1}.

\begin{definition}
    Let $A \in \C^{m \times n}$ be a matrix and let $0 \leq k \leq \min\{m,n\}$.  The \emph{$k$th compound matrix} of~$A$ is the matrix $\calC_k(A) \in \C^{\binom{m}{k} \times \binom{n}{k}}$ whose $(I,J)$ entry (for $|I|=|J|=k$) is the minor of $A$ indexed by row-set~$I$ and column-set~$J$, i.e., $\det(A[I,J])$.  (The rows and columns of $\calC_k(A)$ are considered to be ordered lexicographically.)
\end{definition}
We will use several elementary properties of compound matrices (see, e.g.,~\cite[Chapter~V]{Ait44} and \cite[Chapter 2, Lemma 1.2]{God93}), with Cauchy--Binet being the most important.\footnote{Indeed, \Cref{itm:dagger,itm:diag} below are trivial, \Cref{itm:conj,itm:unitary} follows easily from these given Cauchy--Binet, and \Cref{itm:charpoly} then follows (at least for diagonalizable~$A$) by establishing that $\tr \calC_k(A)$ is the $k$th elementary symmetric polynomial of~$A$'s eigenvalues.}
\begin{fact}                                        \label{fact:compounds}
    The following hold for all matrices of appropriate shape:
    \begin{enumerate}
        \item \emph{(Generalized Cauchy--Binet.)} $\calC_k(AB) = \calC_k(A) \calC_k(B)$.
        \item \label{itm:dagger} $\calC_k(A^\dagger) = \calC_k(A)^\dagger$.
        \item \label{itm:diag} If $A = \diag(a_1, \dots, a_n)$ is diagonal, then $\calC_k(A)$ is diagonal, with $\calC_k(A)_{J,J} = \prod_{i \in J} a_j$.  In particular, $\calC_k(\Id) = \Id$.
        \item \label{itm:conj} If $Q = \diag(q_1, \dots, q_n)$ is diagonal, then $\calC_k(Q^{\dagger\!} A Q)_{I,J} = \calC_k(Q)^*_{I,I} \cdot \calC_k(A)_{I,J} \cdot \calC_k(Q)_{J,J} = \calC_k(A)_{I,J} \cdot \prod_{i \in I} q_i^* \cdot \prod_{j \in J} q_j$.
        \item \label{itm:unitary} If $A$ is unitary then so too is $\calC_k(A)$.
        \item \label{itm:charpoly} $\charpoly{A}{x} = \sum_{k=0}^n (-1)^k (\tr \calC_k(A)) x^{n-k}$.
    \end{enumerate}
\end{fact}

The Cauchy--Binet Theorem yields a formula for the minor of a matrix product in terms of minors of the multiplicands.  One can also obtain formulas for minors of sums of matrices in terms of minors of the summands.  The two-matrix case is classical; see, e.g., \cite[\S44, Ex.~5]{Ait44} or \cite[Lemma~A.1]{CSS07}.  Determinants of sums of more than two matrices were studied in, e.g.,~\cite{Ami80,RS87}; we quote here a formula whose short proof is given in~\cite{HPS18}:
\begin{proposition}                                     \label{prop:det-sum}
    (\cite[Lemma~3.1]{HPS18}.)  For $n \times n$ matrices $A_1, \dots, A_\atms$,
    \[
        \calC_n(A_1 + \cdots + A_\atms)_{[n],[n]} = \det(A_1 + \cdots + A_\atms)
        = \sum_{\substack{I_1 \sqcup \cdots \sqcup I_\atms = [n] \\
                                     J_1 \sqcup \cdots \sqcup J_\atms = [n] \\
                                     |I_t| = |J_t|\ \forall t \in [\atms]}} \pm \calC_{|I_1|}(A_1)_{I_1, J_1} \cdots \calC_{|I_\atms|}(A_\atms)_{I_\atms,J_\atms},
    \]
    where the $\pm$ sign corresponding to partitions $(I_1, \dots, I_\atms)$ and $(J_1, \dots, J_\atms)$ is equal to $\sgnrep(\sigma)$, where $\sigma \in \symm{n}$ is the permutation taking $I_t$ to $J_t$ for each $t \in [c]$.
\end{proposition}
We will need to consider minors of more complicated expressions than just products or sums of matrices.  On the other hand, we will not need to know precise formulas; just something about their structure.  To that end, we state \Cref{prop:our-needed-formula} below.  A proof of a generalization of \Cref{prop:our-needed-formula} is given in \Cref{app:minors} (and ``unrolling'' that proof would yield \Cref{prop:det-sum}, e.g.).
%We will need the following XXX more general fact,  that each minor of a matrix sum/product such as $A_1 A_2 + B_1B_2B_3B_4 + \cdots + H_1H_2H_3$ is a fixed multilinear combination of the minors of the individual matrices involved.
\begin{proposition}                                     \label{prop:our-needed-formula}
    Let $A^{(1)}_1, \dots, A^{(1)}_{t_1}, A^{(2)}_1, \dots, A^{(2)}_{t_2}, \dots, A^{(\atms)}_1, \dots, A^{(\atms)}_{t_\atms}$ be indeterminate matrices.  For each $A^{(i)}_j$, let $\overline{A}^{(i)}_j$ be a (potentially) augmented matrix, containing~$A^{(i)}_j$ as a submatrix and with all other entries being constants.
    Then each minor
    \[
        \calC_k\parens*{\ol{A}^{(1)}_1 \cdots \ol{A}^{(1)}_{t_1} + \ol{A}^{(2)}_1 \cdots \ol{A}^{(2)}_{t_2} + \cdots + \ol{A}^{(\atms)}_1 \cdots \ol{A}^{(\atms)}_{t_\atms}}_{I,J}
    \]
    is a fixed linear combination (depending only on $k, I, J, t_1, \dots, t_\atms$ and the constants used in forming the augmentations $\ol{A}^{(i)}_j$) of products of the form
    \[
        \calC_\star(A^{(1)}_1)_{\star,\star}\cdots \calC_\star(A^{(1)}_{t_1})_{\star,\star}
        \cdot
        \calC_\star(A^{(2)}_1)_{\star,\star}\cdots \calC_\star(A^{(2)}_{t_2})_{\star,\star}
        \cdots \calC_\star(A^{(\atms)}_1)_{\star,\star}\cdots \calC_\star(A^{(\atms)}_{t_\atms})_{\star,\star}.
    \]
\end{proposition}

We now consider minors of random matrices.
\begin{definition}
    Let $\bA$ and $\bB$ be $\C^{m \times n}$-valued random matrices.  We say they have \emph{matching first moments} if $\E[\bA] = \E[\bB]$; i.e., $\E[\bA_{ij}] = \E[\bB_{ij}]$ for all $i,j$.  We say that they have \emph{matching first compound moments} if, for each~$k$,  $\calC_k(\bA)$ and $\calC_k(\bB)$ have matching first moments.
\end{definition}
We will also define matching \emph{second} compound moments in terms of all possible products of two minors:
\begin{definition}
    Let $\bA$ and $\bB$ be $\C^{m \times n}$-valued random matrices.  We say they have \emph{matching second compound moments} if $\E[\calC_k(\bA)_{IJ}^* \cdot \calC_{k'}(\bA)_{I'J'}] = \E[\calC_k(\bB)_{IJ}^* \cdot \calC_{k'}(\bB)_{I'J'}]$ for all~$k, k', I, J, I', J'$.
\end{definition}

To see an example of matching second compound moments, let us recall some facts from representation theory. By virtue of the Cauchy--Binet Theorem, if $\pi$ is a representation of group~$\Gamma$, then so too $\calC_k \circ \pi$.\footnote{And from \Cref{fact:compounds}, \Cref{itm:unitary}, if $\pi$ is unitary, then so too is $\calC_k \circ \pi$.}  This is known as the \emph{$k$th~exterior power representation},~$\bigwedge^k \pi$.  We follow \cite{HPS18}'s definition of ``Property ($\calP 1$)''.
\begin{definition}  \label{def:Property-P1}
    Let $\pi$ be a $\repdim$-dimensional representation of a group $\Gamma$. We say $(\Gamma,\pi)$ satisfy Property ($\calP 1$) if the representations $\bigwedge^k \pi$ for $0 \le k \le d$ are irreducible and pairwise non-isomorphic.
\end{definition}

\begin{remark}  \label{rem:std-and-sgn-std-P1}
    $(\symm{\nlift+1}, \stdrep)$ satisfies $(\calP 1)$.
\end{remark}

\begin{remark} \label{rem:hyper-P1}
    $(\Hyperoctahedrals{\nlift}, \idrep)$ satisfies $(\calP 1)$, where $\idrep$ is the \emph{defining} representation of $\Hyperoctahedrals{\nlift}$, which maps to matrices with $\{0,\pm 1\}$ entries with exactly one nonzero in each row and column. The reason this is true is that $\Hyperoctahedrals{\nlift}$ is a Coxeter group and $(\calP 1)$ is known to hold for all Coxeter groups. This result is attributed to Robert Steinber by \cite[Chapter 5, \S 2, Exercise 3(d)]{Bour07}. The reader can find a proof in \cite[Section 5.1]{GP00}.\footnote{The definition of $\Hyperoctahedrals{\nlift}$ that the reader can keep in mind is the set of all signed permutation matrices.}
\end{remark}

We also have the following \emph{Grand Orthogonality Theorem}:
\begin{theorem}                                     \label{thm:grand-orthogonality-theorem}
    Let $\pi$ and $\pi'$ be irreducible $\repdim$-dimensional representations of the compact group~$\Gamma$.  Let $\bg \sim \Gamma$ be drawn according to the uniform (Haar) distribution and write $\bA = \pi(\bg)$, $\bB = \pi'(\bg)$.  Then
    \begin{equation}    \label{eqn:got-rhs}
        \E[\bA_{i,j}^* \cdot \bB_{i',j'}] = \delta_{\pi\pi'}\delta_{ii'} \delta_{jj'}  \frac1d,
    \end{equation}
    where $\delta$ is the Kronecker delta and $\delta_{\pi\pi'}$ corresponds to whether or not $\pi$ and $\pi'$ are isomorphic.
\end{theorem}
Using the fact that the right-hand side of \Cref{eqn:got-rhs} does not depend on~$\Gamma$ except through whether $\pi$ and $\pi'$ are isomorphic, we conclude the following:
\begin{corollary}                                       \label{cor:P1}
    Let $(\Gamma,\pi)$ and $(\Gamma',\pi')$ be pairs satisfying $(\calP 1)$.  Write $\bA = \pi(\bg)$ and $\bB = \pi'(\bg')$ for $\bg \sim \Gamma$, $\bg' \sim \Gamma'$ drawn from the uniform (Haar) distribution.  Then $\bA$ and $\bB$ have matching second compound moments.
\end{corollary}

We now come to our main theorem for this section, which generalizes~\cite[Theorem~1.8]{HPS18}.
\begin{theorem}                                       \label{thm:factorpolyinvariance}
    The expected characteristic polynomial $\E[\charpoly{\bcalA}{x}]$ for $\bcalA$ defined in \Cref{eq:random-lift-represented} has the same value for any pair $(\Gamma, \pi)$ satisfying~$(\calP1)$. %\rnote{Is this theorem almost `obvious' from 1.3~Lemma in Godsil for the interp of the coefficients of the charpoly of a graph?}
\end{theorem}
\begin{proof}
    By \Cref{fact:compounds}, \Cref{itm:charpoly}, it suffices to prove the stronger fact that the first compound moments of~$\calA$ are invariant to the choice of $(\Gamma, \pi)$ satisfying~$(\calP1)$.

    We can write each $\pi^{\boxtimes V}(\bQ_i)$ as
    \[
        \pi^{\boxtimes V}(\bQ_j(1), 1, \dots, 1) \cdot \pi^{\boxtimes V}(1, \bQ_j(2),  \dots, 1) \cdots \pi^{\boxtimes V}(1, 1, \dots, \bQ_j(\vtcs)),
    \]
    and each matrix $\pi^{\boxtimes V}(1, \dots, 1, \bQ_j(v), 1, \dots, 1)$ is simply the random matrix $\pi(\bQ_j(v))$ augmented, block-diagonally, by identity matrices.  Thus \Cref{prop:our-needed-formula} can be applied to~$\bcalA$ and we obtain that each minor $\calC_k(\bcalA)_{IJ}$ is a fixed linear combination of products of the form
    \[
        \prod_{i=1}^\atms \calC_{\star}(\pi(\bQ_j(\vtcs))^\dagger)_{\star,\star} \cdots \calC_{\star}(\pi(\bQ_j(1))^\dagger)_{\star,\star}\cdot  \calC_{\star}(A_j)_{\star,\star} \cdot \calC_{\star}(\pi(\bQ_j(1)))_{\star,\star} \cdots \calC_{\star}(\pi(\bQ_j(\vtcs)))_{\star,\star}.
    \]
    Using linearity of expectation and the fact that all matrices $\bQ_j(r)$ are independent, we get that $\E[\calC_k(\bcalA)_{IJ}]$ is a fixed linear combination of products of expectations of the form
    \[
        \E[\calC_{\star}(\pi(\bQ_j(j)^\dagger))_{\star,\star} \cdot \calC_{\star}(\pi(\bQ_j(j)))_{\star,\star}].
    \]
    But each such expectation is invariant to the choice of $(\Gamma, \pi)$ satisfying~$(\calP1)$, by \Cref{cor:P1}.
\end{proof}

We derive the following corollary with a short proof.
\begin{corollary}   \label{cor:expected-factor-poly-lift}
    Let $H = A_1 + \cdots + A_\atms$ be a sum graph and let $\bG$ be a random additive $\nlift$-lift. Following the proof of \Cref{prop:lift-is-quotient}, we can write $\bG$ as a sum graph $\bA_{1,1}+\cdots+\bA_{1,\nlift} + \cdots + \bA_{\atms,1}+\cdots+\bA_{\atms,\nlift}$.
    Suppose we choose $\Gamma = \symm{\nlift+1}$ and $\pi$ as its standard representation. Then
    \begin{align} \label{eq:expected-factor-poly}
        \E[\charpoly{\bcalA}{x}] = \E[\factorpoly{\bA_{1,1},\cdots,\bA_{\atms,\nlift}}{x}]
    \end{align}
    for $\bcalA$ from \Cref{eq:random-lift-represented}.
\end{corollary}
\begin{proof}
    From \Cref{thm:factor-poly-comes-from-random-edge-signing}, the right hand side can be rewritten as as $\E[\charpoly{\widetilde{\bcalA}}{x}]$ where $\widetilde{\bcalA}$ is set according to \Cref{eq:random-lift-represented} with the $(\Gamma,\pi)$ pair chosen as $(\Hyperoctahedrals{\nlift},\idrep)$. The equality immediately follows from combining \Cref{rem:std-and-sgn-std-P1} and \Cref{rem:hyper-P1} with \Cref{thm:factorpolyinvariance}.
\end{proof}

Combining the above with \Cref{thm:root-bound} and the fact that \factorpolynomial of a sum graph is monic, we obtain
\begin{corollary}       \label{cor:root-bound-general}
    Suppose $H=A_1+\cdots+A_\atms$ is a connected sum graph, $\bQ_i$ and $\bcalA$ are the same as in \Cref{eq:random-lift-represented}, and $(\Gamma,\pi)=(\symm{\nlift},\stdrep)$. Then
    $\E[\charpoly{\bcalA}{x}]$ has all its roots in the interval $[-\specrad(X),\specrad(X)]$ where $X = A_1 \addprod \cdots \addprod A_\atms$.
\end{corollary}
\begin{proof}
    From \Cref{cor:expected-factor-poly-lift}, it suffices to show that all roots of the RHS of \Cref{eq:expected-factor-poly} lie in the desired interval. Since the \factorpolynomial is always monic, it is enough to show that for a fixed lift of $H$, called $G$ given by $A_{1,1}+\cdots+A_{\atms,\nlift}$, $\factorpoly{A_{1,1},\dots,A_{\atms,\nlift}}{x}$ has all its roots in $[-\specrad(X),\specrad(X)]$. This follows immediately from combining \Cref{rem:connected-comp-lift-is-quotient} and \Cref{rem:root-bound-factor}.
\end{proof}

\section{$X$-Ramanujan Lifts and Proof of \Cref{thm:additive-quasi}}
In this section, we bring all the tools developed in the previous sections together and prove a statement from which our main theorem follows.

\begin{theorem}
Suppose $H = A_1 + \cdots + A_\atms$ is a sum graph on finite vertex set $V$. Then for every $\nlift\in\N^+$, there is an additive $\nlift$-lift $G$ of $H$ such that every new eigenvalue $\lambda\in\spec(G)\setminus\spec(H)$ has $\lambda \le \specrad(A_1\addprod\cdots\addprod A_\atms)$.
\end{theorem}
\begin{proof}
    Let $\bQ_1,\ldots,\bQ_c$ be uniformly random potentials.
    It follows from \Cref{fact:perm-from-transpo} that there is a product distribution $\calD$ over $\LiftEnc_{\nlift,V}^\atms$ such that each $\bQ_i$ is distributed as $\StringToPotential(\bz_i)$ for $\bz\sim\calD$.
    Thus, we have
    \[
        \LiftAdj{\bcalQ}{\stdrep}(H) = \LiftAdj{\bz}{\stdrep}(H)
    \]
    Recall that from \Cref{thm:lifts-interlacing-family}, $\left\{\charpoly{\LiftAdj{z}{\stdrep}(H)}{x}\right\}_{z\in\LiftEnc_{\nlift, V}^\atms}$ is an interlacing family, which means $\E_{\bz\sim\calD}\left[\charpoly{\LiftAdj{\bz}{\stdrep}(H)}{x}\right]$ is real-rooted and from \Cref{thm:root-bound-interlace} we have the existence of $z^*$ such that
    \[
        \maxroot\left(\charpoly{\LiftAdj{z^*}{\stdrep}(H)}{x}\right) \le \maxroot\left(\E\left[\charpoly{\LiftAdj{\bcalQ}{\stdrep}(H)}{x}\right]\right)
    \]

    By \Cref{cor:root-bound-general}, the right hand side of the above expression is at most $\specrad(A_1\addprod\cdots\addprod A_\atms)$.

    The new eigenvalues of the lift given by $z^*$, namely $\spec(G)\setminus\spec(H)$, are exactly given by the roots of $\charpoly{\LiftAdj{z^*}{\stdrep}(H)}{x}$ from \Cref{rem:spectrum-of-lift} and hence the theorem follows.

\end{proof}

\begin{proof}[Proof of \Cref{thm:additive-quasi}]
    Since $X$ is an additive product graph, it can be written as $A_1\addprod\cdots\addprod A_c$ for graphs $A_1,\ldots,A_c$ on finite vertex set $V$. Let $G_1 = A_1 + ... + A_c$, and let $G_n$ be an additive $n$-lift of $G_1$ for which all new eigenvalues are bounded by $\specrad(X)$. Since $\specrad(X)<\chi(X)$ by assumption, $G_n$ is connected. Thus, each $G_n$ is a quotient of $X$ by \Cref{prop:lift-is-quotient}, and has at most $|V|$ eigenvalues that exceed $\specrad(X)$. Hence $X$ is $|V|$-quasi-Ramanujan.
\end{proof}

\subsection*{Acknowledgments}
We thank Nikhil Srivastava and Marc Potters for providing us with \Cref{thm:finite-free-root-bound} and its proof in a personal communication.  We thank Alex Lubotzky for providing us with some references concerning generalized Ramanujan graphs.  We thank Mohan Ravichandran for bringing~\cite{Rav16,LR18} to our attention.  We thank Peter Sarnak for relevant discussions and for pointing us to~\cite{KFSH19}.

\bibliographystyle{alpha}
\bibliography{main}

\appendix

\section{Formulas for minors} \label{app:minors}

\begin{definition}
    We define a \emph{matrix formula} to be a tree in which the leaves are labeled with distinct formal symbols for matrices, and the internal nodes are of three types: \emph{product}, \emph{sum}, and \emph{augmentation}.   A  product (respectively, sum) gate has two or more ordered children, and computes the product (respectively, sum) of its children.  (It will be convenient to order the children of a sum node, despite commutativity of matrix summation.)  An augmentation node has only one child, but comes together with some fixed additional rows and/or columns; it operates by extending its child with these rows/columns.
\end{definition}
\begin{remark}
    \myfig{.5}{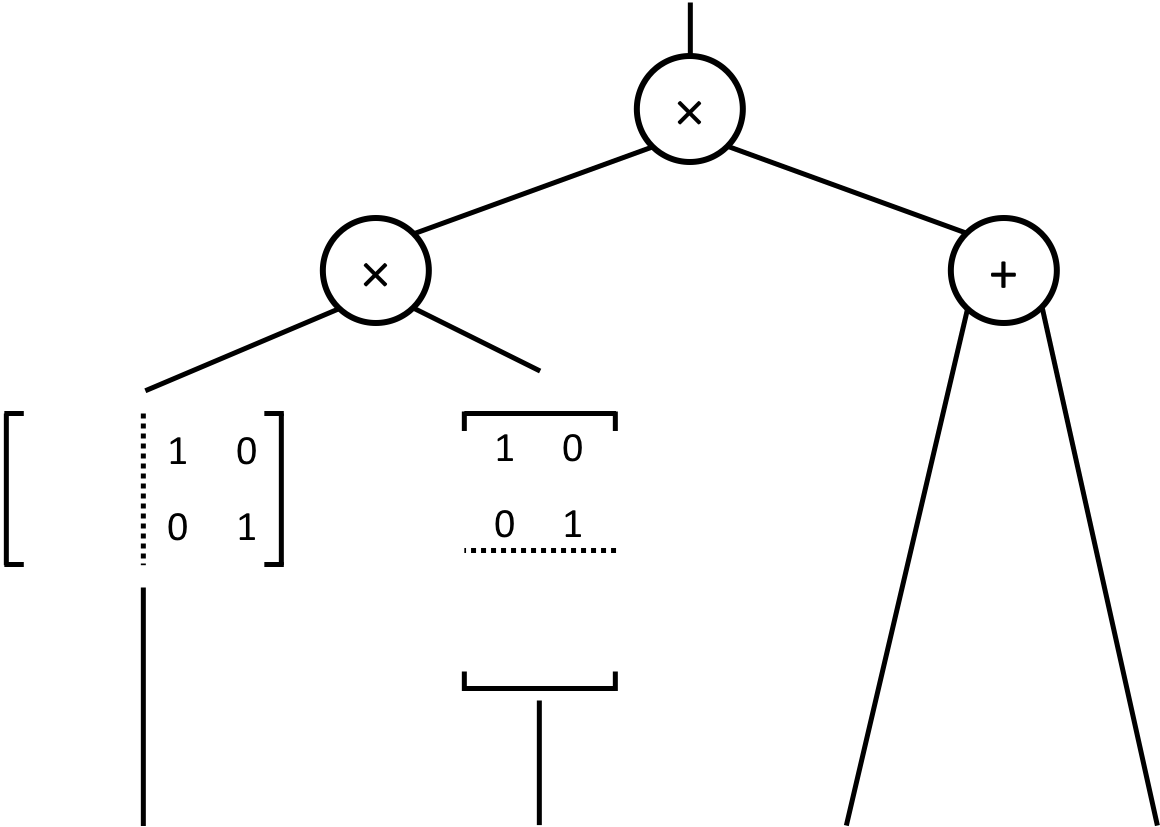}{An example matrix formula $\Phi$}{fig:matrix-formula-eg}

    As an example, the matrix formula~$\Phi$ shown in \Cref{fig:matrix-formula-eg} computes the formula
    \[
    \Phi(A_1, A_2, A_3, A_4) =
    \left[
    \begin{array}{cc;{2pt/2pt}cc}
         &  & 1 & 0 \\
 \multicolumn{2}{c;{2pt/2pt}}{\smash{\raisebox{.5\normalbaselineskip}{$A_{1}$}}}
             & 0 & 1
    \end{array}
    \right]
    \left[
    \begin{array}{cc}
         1 & 0 \\
         0 & 1 \\  \hdashline[2pt/2pt]
 \multicolumn{2}{c}{\smash{\raisebox{-.5\normalbaselineskip}{$A_{2}$}}} \\
            &
    \end{array}
    \right]
    \parens*{A_3 + A_4}.
    \]
    It is tacitly assumed that the matrix dimensions are always appropriate for the operations involved; in our example, $A_1$ should be $2 \times m$, $A_2$ should be $m \times 2$, and $A_3, A_4$ should be $2 \times n$, for some $m,n$.  Incidentally, if $m = 2$ in this example, then $\Phi$ computes $(A_1 + A_2)(A_3 + A_4)$.  We will generally identify a matrix formula $\Phi$ with the function of matrices it computes; we also always list its leaves/arguments in left-to-right order.
\end{remark}

The main result of this section, which includes \Cref{prop:our-needed-formula} as a special case, is the following:
\begin{proposition}                                     \label{prop:matrix-minor}
    Let $\Phi(A_1, \dots, A_t)$ be a matrix formula.  Then any compound matrix entry $\calC_k(\Phi(A_1, \dots, A_t))_{I,J}$ is a linear combination of products of the form
    \[
        \calC_\star(A_1)_{\star,\star} \cdot \calC_\star(A_2)_{\star,\star} \cdots \calC_\star(A_t)_{\star,\star},
    \]
    where the coefficients in the linear combination, as well as the values replacing the $\star$ symbols, depend only on the structure of $\Phi$ and the row/column constants in its augmentation nodes.
\end{proposition}
% \rnote{commented out the 2-matrix-sum example}
%A very simple example of \Cref{prop:matrix-minor} is the formula for the determinant of the sum of two $n \times n$ matrices,
%\[
%    \calC_n(A+B)_{[n],[n]} = \det(A+B) =  \sum_{k=0}^n \sum_{\substack{I, J \subset [n] \\ |I| = |J| = k} } \eps_{I,J} \cdot \calC_k(A)_{I,J} \cdot \calC_{n-k}(B)_{I^c, J^c}
%\]
%where $\eps_{I,J} = \sgnrep(\sigma)$ with $\sigma \in \symm{n}$ being the permutation taking $I,I^c$ to~$J,J^c$.\rnote{One needs to also specify that elements are in lexicographical order, but since we don't need this, I didn't want to spend extra space on this.}
\begin{proof}
    We may assume all product and sum nodes have fan-in exactly~$2$, by expanding them to multiple fan-in-$2$ nodes.  We can also convert all sum nodes to combinations of product and augmentation nodes by using the identity $A + B = \left[
    \begin{array}{c;{2pt/2pt}c}
         A &  \Id
    \end{array}
    \right]
\left[
    \begin{array}{c}
         \Id \\ \hdashline[2pt/2pt]  B
    \end{array}
    \right]
    $, for appropriately shaped identity matrices~$\Id$.  Finally, we may assume all augmentation nodes only add one row or column, by expanding them to multiple augmentation nodes.  We can now prove the proposition by structural induction on the formula~$\Phi$, the base case of a single leaf being obvious.  By our simplifications, there are only two inductive cases: the product of two formulas, and the augmentation of a formula by a single row or column.

    In the product case we can use Cauchy--Binet:
    \begin{align}
        \calC_k(\Phi_1(A_1, \dots, A_s) \cdot \Phi_2( B_1, \dots, B_t))_{I,J} &= (\calC_k(\Phi_1(A_1, \dots, A_s)) \cdot \calC_k(\Phi_2( B_1, \dots, B_t)))_{I,J} \nonumber \\
        &= \sum_{K} \calC_k(\Phi_1(A_1, \dots, A_s))_{I,K} \cdot \calC_k(\Phi_2(B_1, \dots, B_t))_{K,J}. \label{eqn:CB1}
    \end{align}
    By induction,  each $\calC_k(\Phi_1(A_1, \dots, A_s))_{I,K}$ is a linear combination of products $\calC_\star(A_1)_{\star,\star} \cdots \calC_\star(A_s)_{\star,\star}$, and similarly for each $\calC_k(\Phi_2(B_1, \dots, B_t))_{K,J}$.  Thus \Cref{eqn:CB1} establishes that $\calC_k(\Phi_1 \cdot \Phi_2)_{I,J}$ is a linear combination of products $\calC_\star(A_1)_{\star,\star} \cdots \calC_\star(A_s)_{\star,\star} \cdot \calC_\star(B_1)_{\star,\star} \cdots \calC_\star(B_t)_{\star,\star}$.

    As for the augmentation case, suppose we are considering $\calC_k(\Phi)_{I,J}$ where  $\Phi =\left[
    \begin{array}{c;{2pt/2pt}c}
         c &  \Phi_1
    \end{array}
    \right]$ with $c$ being a fixed column vector. (The case of augmentation on the other side, or of augmentation by a row vector, is essentially the same.)  Now $\calC_k(\Phi)_{I,J}$ is either a minor of $\Phi_1$ (if $1 \not \in J$), in which case we're done by induction, or else it is of the form $\det\parens*{
    \left[
    \begin{array}{c;{2pt/2pt}c}
         c_I &  \Phi'_1
    \end{array}
    \right]}$ for some submatrix $\Phi'_1$ of $\Phi_1$.  Then by cofactor expansion of this determinant along the column~$c_I$, we obtain a linear combination of minors of~$\Phi'_1$.  In turn, this is a linear combination of minors of~$\Phi_1$, so again we're done by induction.
\end{proof}

\section{Ramanujan quotients of vertex transitive free products}   \label{sec:FFC-appendix}
\begin{definition}
    Given graphs $G_1, \dots, G_\freeprods$ rooted at $v_1, \dots v_\ell$ respectively, consider the following construction. Define $\widetilde{V}$ as the disjoint union of all $V'(G_i) := V(G_i)\setminus\{v_i\}$ and for each $v \in \widetilde{V}$, let $\Type(v)$ denote the unique $j$ such that $v \in V'(G_j)$. Let $V$ be the set of finite strings of the form $a_1 a_2 \dots a_k$ such that $a_1 \in \widetilde{V}$ and for $i > 1$, $a_i \in \widetilde{V} \setminus V'(G_{\Type(a_{i-1})})$. For $v \in V$ and $x, y \in \widetilde{V}$ such that $vx$ and $vy$ are also in $V$, we put an edge between $vx$ and $vy$ if $\Type(x) = \Type(y)$ and $\{x, y\}$ is an edge in $G_{\Type(x)}$; and we put an edge between $v$ and $vx$ if $\{v_{\Type(x)}, x\}$ is an edge in $G_{\Type(x)}$. We call the graph obtained by placing edges according to the described rule between vertices of $V$ as the \emph{free product} of $G_1, \dots, G_\freeprods$, denoted as $G_1 * \cdots * G_\freeprods$.
\end{definition}

\begin{remark}
    The free product of Cayley graphs rooted at the identity of the corresponding group coincides with the Cayley graph of the free product of the corresponding groups.  The free product of $d$ rooted edges, where a rooted edge can be thought of as the Cayley graph of $\F_2$ rooted at $0$, is the $d$-regular infinite tree.
\end{remark}

Next, we state results of how the spectrum of the free product of a collection of graphs is related to the spectra of those graphs in the collection.

\begin{definition}
    Given a measure $\mu$ on $\R$, the \emph{Cauchy transform} of $\mu$ is
    \[
        \cauchy{\mu}{x} := \int_{\R} \frac{1}{x-\lambda} d\mu(\lambda)
    \]
    For a real-rooted polynomial $p$, we will often consider the measure $\mu_p$ which is the uniform distribution over its roots. When considering this measure, we abuse notation and write $\cauchy{p}{x}$ instead of $\cauchy{\mu_p}{x}$.
\end{definition}

\begin{definition}
    The \emph{inverse Cauchy transform} of $\mu$ is
    \[
        \invcauchy{\mu}{x} := \sup\{y:\calG_{\mu}(y) = x\}
    \]
    when the quantity is well defined.
\end{definition}

The free convolution operator from free probability theory (see e.g. \cite{Spe09}) gives the distribution of the sum of free random variables from the distributions of the individual random variables.
\begin{definition}
    Given measures $\mu_1, \mu_2, \dots, \mu_\freeprods$, the \emph{free convolution} $\mu_1 \sqsum \cdots \sqsum \mu_\freeprods$ is defined as the measure satisfying
    \[
        \invcauchy{\mu_1 \sqsum \cdots \sqsum \mu_\freeprods}{x} = \sum_{i=1}^\freeprods \invcauchy{\mu_i}{x} - \frac{\freeprods - 1}{x}
    \]
\end{definition}

The significance of the free convolution to our work is that the spectral measure of the free product of $\freeprods$ graphs is the free convolution of the spectral measures of those graphs. In particular, one can find the following in \cite[Theorem 9.19]{Woess00}.
\begin{theorem}
    Given vertex transitive graphs $G_1, \dots, G_\ell$ rooted at $v_1, \dots, v_\ell$ with spectral measures $\mu_1, \dots, \mu_\freeprods$, the spectral measure of the free product graph $G_1 * \cdots * G_\freeprods$ is given by $\mu_1\sqsum \cdots \sqsum \mu_\freeprods$.
\end{theorem}

\begin{definition}
    Given degree-$d$ polynomials, $p$ and $q$, let $A$ and $B$ be $d\times d$ diagonal matrices with the roots of $p$ and $q$ respectively on the diagonal. Let $\bQ$ be a random orthogonal matrix drawn from the Haar distribution over the group of $d\times d$ orthogonal matrices under multiplication. We define
    \[
        p \sqsum_d q := \E \left[\charpoly{\bQ A\bQ^\dagger + B}{x}\right]
    \]
    as the \emph{finite free convolution} of $p$ and $q$.
\end{definition}

Suppose $G_1, \dots, G_\freeprods$ are vertex transitive graphs where $G_i$ is $d_i$-regular and $\specrad(X)$, the spectral radius of $X$, is at most $\sum_{i=1}^\freeprods d_i$ where $X = G_1 \ast \cdots \ast G_\freeprods$. We show the existence of an infinite family of $X$-Ramanujan quotients. Let $s = \mathsf{LCM}(|V(G_1),\dots,|V(G_\freeprods)|)$. To construct an $X$-Ramanujan quotient on $st$ vertices, an approach is to take $A_1,\dots,A_\freeprods$ where $A_i$ is a graph comprising of $\frac{st}{|V(G_i)|}$ disjoint copies of $G_i$ and take $\bH = \sum \bP_i A_i \bP_i^\dagger$ where each $\bP_i$ is a random permutation matrix. The existence of an infinite family of $X$-Ramanujan graphs immediately follows from the theorem below.
\begin{theorem}     \label{thm:H-is-X-Ramanujan}
    $\bH$ is $X$-Ramanujan with positive probability.
\end{theorem}

Before proving \Cref{thm:H-is-X-Ramanujan} we state some results from \cite{MSS15d,MSS15c} and a result that is from personal communication with Nikhil Srivastava and Marc Potters. We overload notation and use $\lambda_k(A)$ to denote the $k$-th largest eigenvalue of matrix $A$ and $\lambda_k(p)$ to denote the $k$-th largest root of polynomial $p$. Another observation we point out is that since each $G_i$ is vertex transitive, it is regular with degree $d_i$.
\begin{lemma}[Theorem 3.4 of \cite{MSS15d}]     \label{lem:interlace}
    $\lambda_k(\bH) \le \lambda_k(\E \left[\charpoly{\bH}{x}\right])$ with positive probability.
\end{lemma}

\begin{lemma}[Corollary 4.9 of \cite{MSS15d}]   \label{lem:quadrature}
    \[
        \E\left[\charpoly{\bH}{x}\right] = \left(x-\sum_{i=1}^\freeprods d_i\right) p_1(x) \sqsum_{d-1}\cdots\sqsum_{d-1} p_\freeprods(x)
    \]
    where $p_i(x) = \frac{\charpoly{A_i}{x}}{x-d_i}$.
\end{lemma}

\begin{lemma}[Theorem 1.7 of \cite{MSS15c}]
    \[
        \invcauchy{q_1\sqsum_{d-1} \cdots~\sqsum_{d-1}q_\freeprods}{x} \le \sum_{i=1}^\freeprods \invcauchy{q_i}{x} - \frac{\ell-1}{x} = \invcauchy{q_1\sqsum~\cdots~\sqsum q_\freeprods}{x}
    \]
\end{lemma}

The following theorem and its proof was given to us by Nikhil Srivastava and Marc Potters.
\begin{theorem} \label{thm:finite-free-root-bound}
Let $p_1(x),\dots,p_\ell(x)$ be degree $d$ polynomials.  Let $\mu_1,\dots,\mu_{\ell}$ be the normalized counting measures of the roots of $p_1, \dots, p_\ell$ respectively.  
Then the max root of $p_1 \boxplus_d \dots \boxplus_d p_d$ is upper bounded by the right edge of $\mu_1 \boxplus \dots \boxplus \mu_d$.
\end{theorem}

\begin{proof}
Let $M$ be the right edge of the support of $\mu_1 \boxplus \dots \boxplus \mu_d$.  By the pinching lemma, we have
\begin{align*}
\invcauchy{p_1 \boxplus_d \dots \boxplus_d p_\ell}{w} &\le  \calK_{p_1}(w) + \dots + \calK_{p_\ell}(w) - 1/w \\
&= \calK_{\mu_1 \boxplus \dots \boxplus \mu_\ell} (w) - 1/w
\end{align*}
The proposition now allows us to translate bounds on the inverse Cauchy transform into bounds on the edge of the spectrum.  We apply the proposition, and use the pinching lemma applied to $w = G_{\mu \boxplus \nu}(M)$:
\begin{align*} 
\text{maxroot}(p_1 \boxplus_d\dots \boxplus_d p_{\ell}) &= \calK_{p_1 \boxplus_d \dots \boxplus_d p_{\ell}}(\infty) \\
&\le \calK_{p_1 \boxplus_d\dots\boxplus_d p_\ell}(\calG_{\mu_1 \boxplus\dots \boxplus\mu_\ell}(M)) \\
&\le \calK_{\mu_1 \boxplus \dots \boxplus \mu_\ell} (\calG_{\mu_1 \boxplus \dots \mu_\ell}(M)) - \frac{1}{\calG_{\mu_1 \boxplus \dots \boxplus \mu_\ell}(M)} \\
&= M - \frac{1}{\calG_{\mu_1 \boxplus \dots\boxplus \mu_\ell}(M)} \\
&\le M
\end{align*}
This is the desired result.
\end{proof}

\begin{proof}[Proof of \Cref{thm:H-is-X-Ramanujan}]
    From \Cref{thm:finite-free-root-bound}
    \begin{align*}
        \lambda_1(p_1\sqsum_{d-1}\cdots~\sqsum_{d-1}p_\freeprods) &\le M
    \end{align*}
    where $M$ is the right edge of the support of $\mu_1\boxplus\dots\boxplus\mu_\ell$, also equal to $\specrad(X)$.

    From the assumption that $\specrad(X)<\sum_{i=1}^\freeprods d_i$ and \Cref{lem:quadrature}, it follows that $\lambda_2\left(\E\left[\charpoly{\bH}{x}\right]\right)\le\specrad(X)$. Consequently, from \Cref{lem:interlace} we can conclude that $\lambda_2\le\specrad(X)$ with positive probability, thereby proving the theorem.
\end{proof}

\end{document}